\documentclass[12pt, letterpaper]{article}
\usepackage{dirtytalk}
\usepackage{amsmath}
\usepackage{amsfonts}
\usepackage{amssymb}
\usepackage{amsthm}
\usepackage{breqn}
\usepackage{setspace}
\usepackage{fullpage}
\usepackage{enumitem}
\usepackage{comment}
\usepackage{hyperref}
\usepackage{bbm}
\usepackage{tikz}
\usepackage{enumitem}
\usepackage{mathtools} 
\usepackage[utf8]{inputenc}
\usepackage[english]{babel}
\usepackage{mathtools}
\usetikzlibrary{patterns,arrows,decorations.pathreplacing}
\newtheorem{lemma}{Lemma}
\newtheorem{theorem}{Theorem} 
 
\newtheorem{definition}{Definition}
\newtheorem{claim}{Claim}
\newtheorem{case}{Case}
\newtheorem{conjecture}{Conjecture}

\newtheorem{remark}{Remark}

\newcommand{\N}{\mathcal{N}}

\usepackage[colorinlistoftodos]{todonotes}

\newcommand{\RNum}[1]{\uppercase\expandafter{\romannumeral #1\relax}}
\DeclareMathOperator{\ex}{ex}

\usepackage{authblk}
\begin{document}
\title{Generalized outerplanar Tur\'an number of short paths}
\author[1,\ 2]{Ervin Gy\H{o}ri} 
\author[2,\ 3]{Addisu Paulos}
\author[2]{Chuanqi Xiao}
\affil[1]{Alfr\'ed R\'enyi Institute of Mathematics, Budapest\par
\texttt{gyori.ervin@renyi.hu}}
\affil[2]{Central European University, Budapest\par
\texttt{chuanqixm@gmail.com}}
\affil[3]{Addis Ababa University, Addis Ababa\par
\texttt{addisu_2004@yahoo.com, addisu.wmeskel@aau.edu.et}}
\date{}
\maketitle
\begin{abstract}
Let $H$ be a graph. The generalized outerplanar Tur\'an number of $H$,  denoted by $f_{\mathcal{OP}}(n,H)$, is the maximum number of copies of $H$ in an $n$-vertex outerplanar graph. Let $P_k$ denotes a path on $k$ vertices. In this paper we give an exact value of $f_{\mathcal{OP}}(n,P_4)$ and a best asymptotic value of $f_{\mathcal{OP}}(n,P_5)$. Moreover, we characterize all outerplanar graphs containing $f_{\mathcal{OP}}(n,P_4)$ copies of $P_4$.   
\end{abstract}
\section{Introduction}
In 1941, Tur\'an~\cite{turan} proved a classical result in the field of extremal graph theory. He determined exactly the maximum number of edges an $n$-vertex $K_r$-free graph may contains. After his result, for a graph $H$, the maximum number of edges in an $n$-vertex $H$-free graph, denoted by $\text{ex}(n,H)$, is named as \textit{Tur\'an number of $H$}. A major breakthrough in the study of the Tur\'an number of graphs came in 1966, with the proof of the famous theorem by Erd\H{o}s, Stone and Simonovits~\cite{erd2, erd1}. They determined an asymptotic value of the Tur\'an number of any non-bipartite graph $H$. In particular, they proved  $\ex(n,H)=\left(1-\frac{1}{\chi(H)-1}\right){n\choose 2}+o(n^2)$, where $\chi(H)$ is the chromatic number of $H$. 

Since then, researchers have been interested working on Tur\'an number of class of bipartite (degenerate) graphs and extremal graph problems with some more generality. Determining the maximum number of copies of $H$ in an $n$-vertex $F$-free graph, denoted by $\text{ex}(n,H,F)$, is among such problems. Since we count the number of copies of a given graph which is not necessarily an edge, such an extremal graph problem is commonly named as \textit{generalized Tur\'an problem}. 
The results on $\text{ex}(n,K_r,K_t)$ by Zykov~\cite{zykov} (and independently by Erd\H{o}s~\cite{erd3}), $\ex(n,C_5,C_3)$ by Gy\H{o}ri~\cite{gyori} and $\text{ex}(n,C_3,C_5)$ by Bollob\'as and Gy\H{o}ri~\cite{bolgy} were sporadic initial contributions.

Recently, some researchers were interested in extremal graph problems in some particular family of graphs, for instance, the family of planar graphs.

The study of generalized extremal problems in the family of planar graphs was initiated by Hakimi and Schmeichel~\cite{haksh} in 1979. Define the \textit{generalized planar Tur\'an number} of a graph $H$, denoted by $f_{\mathcal{P}}(n,H)$, as the maximum number of copies of $H$ in an $n$-vertex planar graph. Hakimi and Schmeichel~\cite{haksh} determined the exact value of  $f_{\mathcal{P}}(n,C_3)$ and $f_{\mathcal{P}}(n,C_4)$. Currently, this topic is active and many exact and best asymptotic values were obtained for different planar graphs. We refer to~\cite{aloncaro, gyori3, gyori2} for more results. 

In a different setting, Matolcsi and  Nagy in~\cite{nagy} initiated the study of the generalized planar Tur\'an number version in the family of outerplanar graphs. Before discussing the problem and our findings, some important notations we have used are mentioned below.  

Let $G$ be a graph. The vertex and edge sets of $G$ are denoted respectively by $V(G)$ and $E(G)$. For a vertex $v\in V(G)$, the degree of $v$ is denoted by $d_G(v)$.  We use the notation $d(v)$ instead of $d_{G}(v)$ if there is no ambiguity on the graph. The maximum and minimum degree of $G$ are denoted by $\Delta(G)$ and $\delta(G)$ respectively. 
We denote a $k$-vertex path with vertices $v_1,v_2,\dots, v_k$ in sequential order by $(v_1,v_2,\dots,v_k)$. We may call a $k$-vertex path as a path of length $k-1$, or simply a $(k-1)$-path. We use the notation $P_k$ to denote a $k$-vertex path. Let $H$ be a graph. The notation $\mathcal{N}(H,G)$ is the number of isomorphic copies of $H$ as a subgraph in $G$. We use the notation $[k]$ to describe the set of the first $k$ positive integers.

\begin{definition}
The \textit{generalized outerplanar Tur\'an number} 
of a graph $H$, denoted by $f_{\mathcal{OP}}(n,H)$, is the maximum number of copies of $H$ in an $n$-vertex outerplanar graph. i.e.,

\[f_{\mathcal{OP}}(n,H)=\max\{\mathcal{N}(H,G): G\ \text{is} \ n\text{-vertex}\ \text{outerplanar graph} \}.\]
\end{definition}
Matolcsi and  Nagy in~\cite{nagy} determined sharp and asymptotically sharp bounds of $f_{\mathcal{OP}}(n,H)$ for certain families of graphs $H$ and described the extremal graphs. In particular, they determined the exponential growth of the generalized outerplanar Tur\'an number of $P_k$,  as a function of $k$. They also determined the exact value of $f_{\mathcal{OP}}(n,P_3)$ and characterized the extremal graphs.

\begin{theorem}(Matolcsi and  Nagy~\cite{nagy})
$$h(k){n\choose 2}<f_{\mathcal{OP}}(n,P_k)\leq 4^k{n\choose 2}, \ \text{where} \lim\limits_{k\longrightarrow\infty}{\sqrt[k]{h(k)}}=4.$$  
\end{theorem}

\begin{theorem}(Matolcsi and  Nagy~\cite{nagy}) Suppose that $n\geq 3$. Then $$f_{\mathcal{OP}}(n,P_3)=\frac{n^2+3n-12}{2},$$
and the unique extremal outerplanar graph containing $f_{\mathcal{OP}}(n,P_3)$ $P_3$'s 
is $K_1+P_{n-1}$.
\end{theorem}

In this paper, we determine the exact value of $f_{\mathcal{OP}}(n,P_4)$ and an asymptotic value $f_{\mathcal{OP}}(n,P_5)$. Moreover, we characterize all outerplanar graphs containing $f_{\mathcal{OP}}(n,P_4)$ copies of $P_4$ as a subgraph. The following two theorems summarize our results. 
\begin{theorem}\label{t1}
For $n\geq 8$, $$f_{\mathcal{OP}}(n,P_4)=2n^2-7n+2.$$
Moreover, for $n\geq 9$, the only $n$-vertex outerplanar graph containing $f_{\mathcal{OP}}(n,P_4)$ number of $P_4's$ is $K_1+P_{n-1}.$
\end{theorem}

\begin{theorem}\label{tms2}
$$f_{\mathcal{OP}}(n,P_5)=\frac{17}{4}n^2+\Theta(n).$$
\end{theorem}


We state the following two easy lemmas without proof.
\begin{lemma}\label{l2}
Every maximal outerplanar graph with at least $3$ vertices contains a degree 2 vertex.
\end{lemma}
\begin{lemma}
For an $n$-vertex, $n\geq 3$, maximal outerplanar graph $G$, $e(G)=2n-3$.
\end{lemma}
We will use the following lemma in our proof of the main results. It gives a relation how to count the number of $P_4$'s in a given simple graph $G$. Its proof can be seen in~\cite{gyori2}, but we give the proof for completeness.
\begin{lemma}\label{lemma1}
For any simple graph $G$, the number of $P_4$'s in $G$ is
$$\mathcal{N}(P_4,G)=\sum\limits_{\{x,y\}\in E(G)}(d(x)-1)(d(y)-1)-3\N(C_3,G),$$
where $C_3$ is a cycle of length three.
\end{lemma}

\begin{proof}
Consider an edge $\{x,y\}\in E(G)$ and count the number of 3-paths containing $x$ as the second and and $y$ the third vertex of the 3-path. There are $d(x)-1$ possibilities to choose the first vertex and $d(y)-1$ possibilities to choose the last vertex of the path. Since the first and the last vertex of the 3-path need to be different, from the total number of $(d(x)-1)(d(y)-1)$ possibilities we need to subtract the number of triangles containing the edge $\{x,y\}$, which is $d(x,y)$. 

Therefore,
\begin{align*}
\N(P_4,G) &= \sum_{\{x,y\}\in E(G)}\left((d(x)-1)(d(y)-1)-d(x,y)\right)\\& = \sum_{\{x,y\}\in E(G)}(d(x)-1)(d(y)-1)-3\N(C_3,G),
\end{align*}
as each triangle is counted 3 times in the sum.
This completes the proof of Lemma~\ref{lemma1}.
\end{proof}
\section{Generalized outerplanar Tur\'an number of the $P_4$}

\begin{figure}[ht]
\centering
\begin{tikzpicture}[scale=0.2]
\draw[thick](0,20)--(-19.3,-5.2) (0,20)--(-5.2,-19.3)(0,20)--(-10,-17.3)(0,20)--(-17.3,-10)(0,20)--(-14,-14)(0,20)--(-20,0)(0,20)--(-19.3,5.2)(0,20)--(-14,14)(0,20)--(-17.3,10)(0,20)--(-10,17.3);
\draw[thick](0,20)--(10,-17.3)(0,20)--(17.3,-10)(0,20)--(20,0)(0,-20)--(0,20)(0,20)--(17.3,10)(0,20)--(10,17.3);

\draw[thick](0,20)--(19.3,-5.2)(0,20)--(5.2,-19.3)(0,20)--(10,-17.3)(0,20)--(17.3,-10)(0,20)--(14,-14)(0,20)--(20,0)(0,20)--(19.3,5.2)(0,20)--(14,14)(0,20)--(17.3,10)(0,20)--(10,17.3);
\draw[fill=black] (19.3,5.2) circle (15pt);
\draw[fill=black] (-19.3,5.2) circle (15pt);
\draw[fill=black] (-19.3,-5.2) circle (15pt);
\draw[fill=black] (19.3,-5.2) circle (15pt);
\draw[fill=black] (14,14) circle (15pt);
\draw[fill=black] (-14,14) circle (15pt);
\draw[fill=black] (-14,-14) circle (15pt);
\draw[fill=black] (14,-14) circle (15pt);
\draw[fill=black] (5.2,19.3) circle (15pt);
\draw[fill=black] (-5.2,19.3) circle (15pt);
\draw[fill=black] (-5.2,-19.3) circle (15pt);
\draw[fill=black] (5.2,-19.3) circle (15pt);
\draw[thick] (0,0) circle (20cm);
\draw[fill=black] (20,0) circle (15pt);
\draw[fill=black] (-20,0) circle (15pt);
\draw[fill=black] (0,20) circle (15pt);
\draw[fill=black] (0,-20) circle (15pt);
\draw[fill=black] (17.3,10) circle (15pt);
\draw[fill=black] (10,17.3) circle (15pt);
\draw[fill=black] (-17.3,10) circle (15pt);
\draw[fill=black] (-10,17.3) circle (15pt);
\draw[fill=black] (-17.3,-10) circle (15pt);
\draw[fill=black] (-10,-17.3) circle (15pt);
\draw[fill=black] (17.3,-10) circle (15pt);
\draw[fill=black] (10,-17.3) circle (15pt);
\end{tikzpicture}
\caption{An $n$-vertex maximal outerplanar graph $G_n=K_1+P_{n-1}$.}
\label{fig2}
\end{figure}
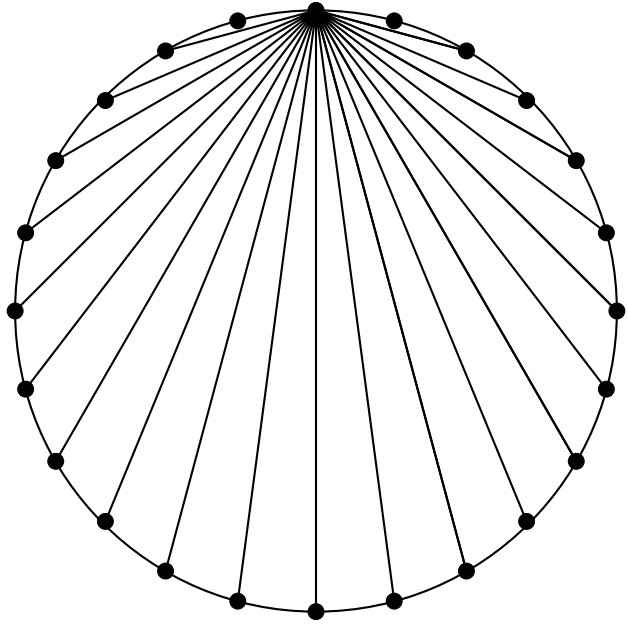

 \begin{proof}[Proof of Theorem~\ref{t1}]
 It is easy to check using Lemma~\ref{lemma1} that an $n$-vertex maximal outerplanar graph $G_n=K_1+P_{n-1}$ shown in Figure~\ref{fig2} contains $2n^2-7n+2$ $P_4$'s. Therefore, $$f_{\mathcal{OP}}(n,P_4)\geq2n^2-7n+2.$$
 
Next we shall prove that $f_{\mathcal{OP}}(n,P_4)\leq2n^2-7n+2.$ Let $G$ be an $n$-vertex maximal outerplanar graph. The main idea of the proof is to find a degree $2$ vertex $w$ in $G$, such that the number of $P_4$'s containing $w$ is at most $4n-9$ and consequently by deleting $w$ and by induction we have 
$$\N(P_4,G)\leq 2(n-1)^2-7(n-1)+2+(4n-9)=2n^2-7n+2.$$
The base case of the induction step,  when $n=7$, is done separately in later in Claim~\ref{c8}.

Now,  let $G$ be a maximal outerplanar graph on $n\geq 8$ vertices, and let $x$ be a degree $2$ vertex in $G$. Since each face (except the unbounded face) in $G$ is a triangle, there exist $u$ and $v$ such that $u,~v$ and $x$ forms a triangle. Let $d(u)=a+2$ and $u_1,u_2,\dots,u_a$ be the vertices which are adjacent to $u$ (except $v$ and $x$) in counterclockwise direction. Let $d(v)=b+2$ and $v_1,v_2,\dots,v_b$ be vertices adjacent to $v$ (except $u$ and $x$) in clockwise direction. Again, since every face of $G$ (except the unbounded face) is a triangle, necessarily the vertices $u_a$ and $v_b$ are identical, see Figure~\ref{figg2}. For the reason that $G$ is a maximal outerplanar graph, $\{u_i,u_{i+1}\}\in E(G)$ and $\{v_j,v_{j+1}\}\in E(G)$, where $i\in[a-1]$ and $j\in [b-1]$, see  the broken red edges in Figure~\ref{figg2}. 

Notice that for $i\in[a-1]$ and $j\in[b-1]$, the arc joining $u_i$ and $u_{i+1}$ in counterclockwise direction and the arc joining $v_j$ and $v_{j+1}$ in clockwise direction may contain an interior vertex. We labeling  such vertices in the following way. For each $i\in[a-1]$, denote the interior vertices (if any) between the vertices $u_i$ and $u_{i+1}$ in counterclockwise direction by $u_i^1,u_i^2,\dots,u_i^{n_i}$. Similarly for each  $j\in[b-1]$, denote the interior vertices (if any) between $v_j$ and $v_{j+1}$ in clockwise direction by $v_j^1, v_j^2, \dots, v_j^{n_j}$. 

\begin{figure}[h]
\centering
\begin{tikzpicture}[scale=0.2]
\draw[black,thick](-5.2,19.3)..controls (-3,17) and (3,17) .. (5.2,19.3);
\draw[thick](-5.2,19.3)--(-17.3,10)(-5.2,19.3)--(-17.3,-10)(-5.2,19.3)--(-10,-17.3)(-5.2,19.3)--(-20,0)(-5.2,19.3)--(0,-20)(-5.2,19.3)--(10,-17.3)(5.2,19.3)--(10,-17.3)(5.2,19.3)--(17.3,-10)(5.2,19.3)--(20,0)(5.2,19.3)--(17.3,10);
\node at (-6,21) {$u$};
\node at (6,21) {$v$};
\node at (0,22) {$x$};
\node at (-12,18) {$u_1$};
\node at (-19,11) {$u_2$};
\node at (-22,0) {$u_3$};
\node at (-19,-11) {$u_4$};
\node at (12,-19) {$u_a=v_b$};
\node at (12,18) {$v_1$};
\node at (19,11) {$v_2$};
\node at (22,0) {$v_3$};
\draw[dashed, red,ultra thick](10,17.3)--(17.3,10)--(20,0)--(17.3,-10)--(10,-17.3)--(0,-20)--(-10,-17.3)--(-17.3,-10)--(-20,0)--(-17.3,10)--(-10,17.3);
\draw[fill=black] (5.2,19.3) circle (15pt);
\draw[fill=black] (-5.2,19.3) circle (15pt);
\draw[thick] (0,0) circle (20cm);
\draw[fill=black] (20,0) circle (15pt);
\draw[fill=black] (-20,0) circle (15pt);
\draw[fill=black] (0,20) circle (15pt);
\draw[fill=black] (0,-20) circle (15pt);
\draw[fill=black] (17.3,10) circle (15pt);
\draw[fill=black] (10,17.3) circle (15pt);
\draw[fill=black] (-17.3,10) circle (15pt);
\draw[fill=black] (-10,17.3) circle (15pt);
\draw[fill=black] (-17.3,-10) circle (15pt);
\draw[fill=black] (-10,-17.3) circle (15pt);
\draw[fill=black] (17.3,-10) circle (15pt);
\draw[fill=black] (10,-17.3) circle (15pt);
\end{tikzpicture}
\label{figg2}
\caption{Structure of a maximal outerplanar graph.}
\end{figure}
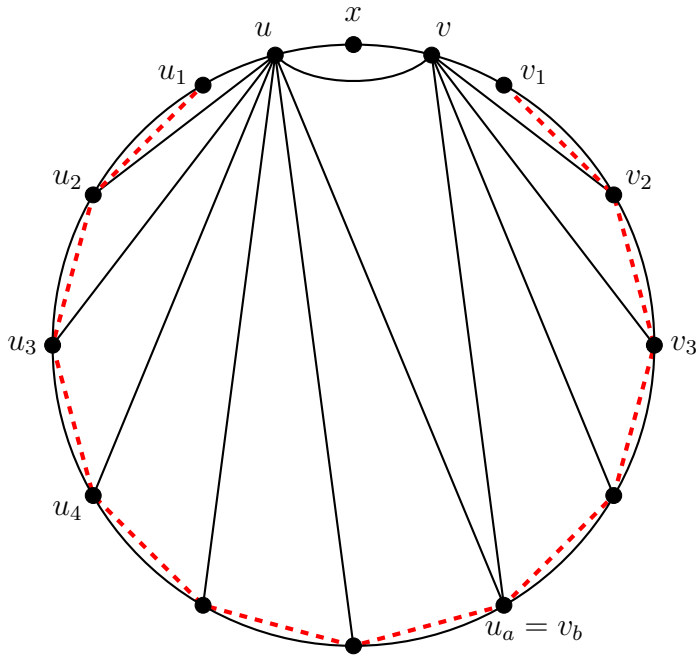
It can be checked that the $3$-paths containing the vertex $x$ are those starting with $(x,u,u_i,\dots)$, $(x,u,v,\dots)$, $(v,x,u,\dots)$, $(x,v,v_i,\dots)$, $(x,v,u,\dots)$ and $(u,x,v,\dots)$. In all cases, vertex $x$ is either an end vertex of the $3$-path or an end vertex of the middle edge of a $3$-paths.
Let $P$ be a $3$-path containing the vertex $x$. We define a \say{\textit{{red edge}}} of $P$ with respect to $x$ as the farthest edge in $P$ from $x$.
Obviously, there are two possibilities based on the position of $x$ in $P$, see Figure~\ref{red}.

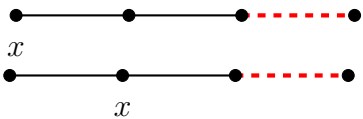
\begin{figure}[h]
\centering
\begin{tikzpicture}[scale=0.15]
\draw[thick](-15,0)--(-5,0)--(5,0);
\draw[dashed, red,ultra thick](5,0)--(15,0);
\draw[fill=black] (-15,0) circle (15pt);
\draw[fill=black] (-5,0) circle (15pt);
\draw[fill=black] (5,0) circle (15pt);
\draw[fill=black] (15,0) circle (15pt);
\node at (-15,-3) {$x$};
\end{tikzpicture}

\begin{tikzpicture}[scale=0.15]
\draw[dashed, red,ultra thick](5,0)--(15,0);
\draw[fill=black] (-15,0) circle (15pt);
\draw[fill=black] (-5,0) circle (15pt);
\draw[fill=black] (5,0) circle (15pt);
\draw[fill=black] (15,0) circle (15pt);
\draw[thick](-15,0)--(-5,0)--(5,0);
\node at (-5,-3) {$x$};
\end{tikzpicture}
\caption{Red edges of a $3$-path $P$ with respect to a vertex $x$. }\label{red}
\end{figure}

Next, we count the number of $3$-paths containing $x$ in $G$ considering the number of possibilities that an edge in $G$ can be a red edge. The following gives us the counting lists for each edge.
\begin{enumerate}
\item The edges $\{x,u\},\{x,v\}$ and $\{u,v\}$ can not be a red edge in a $3$-path that contains $x$. 
\item Each edge $\{u,u_i\}$, $i\in[a-1]$, is a red edge in two $3$-paths containing the vertex $x$. Indeed, $(v,x,u,u_i)$ and $(x,v,u,u_i)$ are the $3$-paths. Similarly each edge $\{v,v_j\}$, $j\in[b-1]$,  is a red edge in two $3$-paths containing the vertex $x$.  
\item Each edge $\{u_i,u_{i+1}\}$, $i\in[a-2]$, is a red edge in two $3$-paths containing the vertex $x$. Indeed, $(x,u,u_i,u_{i+1})$ and $(x,u,u_{i+1},u_i)$ are the $3$-paths. Similarly, each edge $\{v_j,u_{j+1}\}$, $j\in[b-2]$, is a red edge in two $3$-paths containing the vertex $x$. 
\item The edge $\{u,u_a\}$ is a red edge in three $3$-paths containing the vertex $x$. Indeed, the $3$-paths are $(v,x,u,u_a)$, $(x,v,u_a,u)$ and $(x,v,u,u_a)$. Similarly, the edge $\{v,v_b\}$ is a red edge in three $3$-paths containing the vertex $x$. 
\item The edge $\{u_a,u_{a-1}\}$ is a red edge in three $3$-paths containing the vertex $x$. Here it can be checked that the paths are $(x,u,u_a,u_{a-1}),  (x,u,u_{a-1},u_a)$ and $(x,v,u_a,u_{a-1})$. Similarly, the edge $\{v_b,v_{a-1}\}$ is a red edge in three $3$-paths containing the vertex $x$.
\item For $i\in[a-2]$, if there exist interior vertices between $u_i$ and $u_{i+1}$ in counterclockwise direction, then each edge $\{u_{i}^k,u_i^{k+1}\}$, $k\in[n_i-1]$ is not contained in a $3$-path that contains $x$. Moreover, for each vertex $u_i^*$ in $\{u_i^1,u_i^2,\dots,u_i^{n_i}\}$ which is adjacent to $u_{i}$ or $u_{i+1}$, the edge $\{u_{i},u_{i}^*\}$ or $\{u_{i}^*,u_{i+1}\}$ is a red edge in exactly one $3$-path that contains $x$. 
\item For $j\in[b-2]$, if there exist interior vertices between $v_j$ and $v_{j+1}$ in clockwise direction, then each edge $\{v_{j}^k,v_j^{k+1}\}$, $k\in[n_j-1]$ is not contained in a $3$-path that contains $x$. Moreover, for each vertex $v_i^*$ in $\{v_j^1,v_j^2,\dots,v_j^{n_j}\}$ which is adjacent to $v_{j}$ or $v_{j+1}$, the edge $\{v_{j},v_{j}^*\}$ or $\{v_{j}^*,v_{j+1}\}$ is a red edge in one $3$-path that contains $x$. 
\item If there exist interior vertices between $u_{a-1}$ and $u_a$ in counterclockwise direction, then each edge $\{u_a^k,u_a^{k+1}\}$, $k\in[n_a-1]$ is not contained in  $3$-path that contains $x$. Moreover for each vertex $u_a^*\in\{1,2,\dots,n_a\}$ which is adjacent to $u_{a-1}$, the edge $\{u_a,u_a^*\}$ is a red edge in one $3$-path that contains $x$. On the other hand, for each vertex $u_a^*\in\{u_a^1,u_a^2,\dots,u_a^{n_a}\}$ which is adjacent to $u_a$, the edge $\{u_a,u_a^*\}$ is a red edge in two $3$-paths that contains $x$. Indeed that paths are $(x,u,u_a,u_a^*)$ and $(x,v,u_a,u_a^*)$.
\item If there exist interior vertices between $v_{b-1}$ and $v_b$ in counterclockwise direction, then each edge $\{v_b^k,v_b^{k+1}\}$, $k\in[n_b-1]$ is not contained in  $3$-path that contains $x$. Moreover for each vertex $v_b^*\in\{1,2,\dots,n_b\}$ which is adjacent to $v_{b-1}$, the edge $\{v_b,v_b^*\}$ is a red edge in one $3$-path that contains $x$. On the other hand, for each vertex $v_b^*\in\{v_b^1,v_b^2,\dots,v_b^{n_b}\}$ which is adjacent to $v_b$, the edge $\{v_b,v_b^*\}$ is a red edge in two $3$-paths that contains $x$. Indeed that paths are $(x,v,v_b,v_b^*)$ and $(x,u,v_b,v_b^*)$.
 
\end{enumerate}
Notice that the number of edges in an $n$-vertex maximal outerplanar graph is $2n-3$. 

We need the following claims to complete the proof of Theorem \ref{t1}.

\begin{claim}\label{m1}
For each $i\in[a-2]$ and $j\in[b-2]$ there is no interior vertices between $u_{i}$ and $u_{i+1}$ in counterclockwise direction, and no interior vertices between $v_{j}$ and $v_{j+1}$ in clockwise direction. Otherwise, the number of $3$-paths containing $x$ is at most $4n-10$.
\end{claim}
\begin{proof}
 Without loss of generality, assume that there exist interior vertices between $u_{i}$ and $u_{i+1}$ in counterclockwise direction, for some $i\in[a-2]$. We separate the proof into two cases.
 
 \begin{enumerate}
 \item[i.]   There exist at least two interior vertices between $u_{i}$ and $u_{i+1}$ in counterclockwise direction. Then the edge $\{u_i^1,u_i^2\}$ is not contained in any $3$-path containing  $x$. Thus, the number of $3$-paths containing $x$ is at most $2((2n-3)-8)+12=4n-10$.
 
 \item[ii.] There exists one interior vertex between $u_{i}$ and $u_{i+1}$ in counterclockwise direction. Then the edge $\{u_i,u_i^1\}$ and $\{u_{i+1},u_i^1\}$ are contained in only one $3$-path each. Indeed, $(x,u,u_i,u_i^1)$ and $(x,u,u_{i+1},u_i^1)$, respectively. Thus,  the number of $3$-paths containing  $x$ in this case is at most $2((2n-3)-9)+14=4n-10$.
 \end{enumerate}
 This completes the proof of Claim~\ref{m1}.
\end{proof}
\begin{claim}\label{m2}
There is no interior vertex between $u_{a-1}$ and $u_a$.  Moreover, there is no interior vertex between $v_{b-1}$ and $v_{b}$. Otherwise, there exists a degree $2$ vertex in $G$ such that the number of $3$-paths containing this vertex is at most $4n-10$. 
\end{claim}
\begin{proof}
Without loss of generality, we assume that there exist interior vertices between $u_{a-1}$ and $u_a$ in counterclockwise direction. We distinguish two cases considering the number of interior vertices.
\begin{enumerate}
\item[i.] The number of interior vertices between $u_{a-1}$ and $u_a$ in counterclockwise direction is at least 2. In this case the edge $\{u_{a-1}^1,u_{a-1}^2\}$ is contained in no $3$-path containing $x$. Thus, the number of $3$-paths containing $x$ in $G$ is at most $2((2n-3)-8)+12=4n-10$.  

\item[ii.] There exists only one interior vertices between $u_{a-1}$ and $u_a$ in clockwise direction.

It can be checked that $u_{a-1}^1$ is a degree 2 vertex such that the number of $3$-paths containing $u_{a-1}^1$ is at most $4n-10$. Indeed, we interchange the roles of $x$ with $u_{a-1}^1$, $u$ with $u_a$ and $v$ with $u_{a-1}$.

Since $n\geq 8$, either there exists one interior vertex between $u$ and $u_{a-1}$ in counterclockwise direction and one between $v$ and $u_a$ in clockwise direction or there exist at least two interior vertices between $u$ and $u_{a-1}$ in counterclockwise direction or between $v$ and $u_a$ in clockwise direction. We separate the rest proof into $2$ subcases correspondingly.
\begin{enumerate}
\item  There exists one interior vertex $u_{1}$ between $u$ and $u_{a-1}$ in counterclockwise direction and $v_{1}$ between $v$ and $u_{a}=v_{b}$ in clockwise direction, see Figure \ref{gf1}. It can be checked that in this case, $n=8$ and the number of $3$-paths in $G$ is $74$ which is equal to $2n^{2}-7n+2$. 
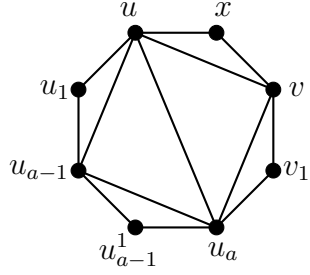
\begin{figure}[h]
\centering
\begin{tikzpicture}[scale=0.07]
\draw[fill=black] (7.7,18.5) circle (40pt);
\draw[fill=black] (18.5,7.7) circle (40pt);
\draw[fill=black] (-7.7,18.5) circle (40pt);
\draw[fill=black] (-18.5,7.7) circle (40pt);
\draw[fill=black] (-18.5,-7.7) circle (40pt);
\draw[fill=black] (-7.7,-18.5) circle (40pt);
\draw[fill=black] (7.7,-18.5) circle (40pt);
\draw[fill=black] (18.5,-7.7) circle (40pt);
\draw[thick](18.5,7.7)--(7.7,18.5)--(-7.7,18.5)--(-18.5,7.7)--(-18.5,-7.7)--(-7.7,-18.5)--(7.7,-18.5)--(18.5,-7.7)--(18.5,7.7);
\draw[thick](-7.7,18.5)--(18.5,7.7)(-18.5,-7.7)--(7.7,-18.5)--(18.5,7.7)(-7.7,18.5)--(-18.5,-7.7)(7.7,-18.5)--(-7.7,18.5);
\node at (9,23) {$x$};
\node at (23,7.7) {$v$};
\node at (23,-7.7) {$v_1$};
\node at (-9,23) {$u$};
\node at (9,-23) {$u_a$};
\node at (-9,-23) {$u_{a-1}^1$};
\node at (-26,-7.7) {$u_{a-1}$};
\node at (-23,7.7) {$u_1$};
\end{tikzpicture}
\caption{An $8$-vertex maximal outerplanar graph.}
\label{gf1}
\end{figure}
\item 
There exist at least two interior vertices between $u$ and $u_{a-1}$ in counterclockwise direction or between $v$ and $u_a$ in clockwise direction. Without loss of generality, we assume that there exist at least two interior vertices between $u$ and $u_{a-1}$ in counterclockwise direction. Hence, all the interior vertices between $u$ and $u_{a-2}$ (counterclockwise direction) are not adjacent to the vertex $u_{a-1}$. Otherwise, it contracts to the fact that $G$ is outerplanar.
It can be seen that the edges $\{u_{a-2},u_{a-3}\}$ and $vx$ are red edges in only one $3$-path containing $u_{a-1}^1$. Indeed the $3$-paths are $(u_{a-1}^1,u_{a-1},u_{a-2},u_{a-3})$ and $(u_{a-1}^1,u_a,v,x)$ respectively. In fact, the vertex $u_{a-3}$ is nothing but $u_1$ if the number of interior vertices between $u$ and $u_{a-1}$ (counterclockwise direction) is two. Notice that the edges $\{u,u_{a-2}\},\{u,u_{a-1}\}, \{u,u_a\}$ and $\{u,v\}$ are red edges in three $3$-paths containing the vertex $u_{a-1}^1$. Therefore the number of $3$-paths containing $u_{a-1}^1$ is at most $2((2n-3)-9)+14=4n-10$. 
\end{enumerate}
\end{enumerate}
This completes the proof of Claim~\ref{m2}.
\end{proof}
\begin{claim}\label{m3}
For each $i\in[a-1]$ and $j\in[b-1]$, if there is no interior vertex between $u_i$ and $u_{i+1}$ and no interior vertex between $v_j$ and $v_{j+1}$. Then, either $\mathcal{N}(P_4,G)=2n^2-7n+2$ or there exists a degree $2$ vertex, say $z$, such that the number of $3$-paths containing $z$ is at most $4n-10$. 
\end{claim}
\begin{proof}
Notice that both $a$ and $b$ are at least 1. We distinguish two cases to finish the proof of the claim.
\begin{enumerate}
\item Either $a=1$ or $b=1$. Without loss of generality assume that $b=1$. In this case $u_a=v_1$ and by Claim~\ref{m1} $u$ is adjacent to every vertices in $G$ and we exactly get the maximal outerplanar graph shown in Figure~\ref{fig2} and hence $\mathcal{N}(P_4,G)=2n^2-7n+2$.
\item $a,b\geq2$. Now consider the degree 2 vertex $u_1$ and the two adjacent vertices $u_2$ and $u$. Next we interchange the roles of $x$ with $u_1$, $u$ with $u_2$ and $v$ with $u$ and check what happens to our previous computation with this choice of degree 2 vertex $u_1$. Since $v$ is adjacent to $u_a$, $u$ can not be adjacent with any interior vertices between $v$ and $u_a$ in clockwise direction. Clearly the vertex $v_1$ is an interior vertex between $v$ and $u_a$ in clockwise direction. Thus, by Claim~\ref{m1}, the number of $3$-paths containing the vertex $u_1$ is at most $4n-10$.  
\end{enumerate}
This completes the proof of Claim~\ref{m3}.
\end{proof}

Considering Claims \ref{m1}, \ref{m2} and \ref{m3}, the only case that $G$ contains $2n^2-7n+2$ number of $3$-paths is when $a=1$ and $v$ is adjacent to every vertex of $G$ or $b=1$ and $u$ is adjacent to every vertex of $G$. Otherwise, there exists a vertex of degree 2, say $v$, such that the number of $3$-paths containing $v$ is at most $4n-10$. Hence, by deleting $v$ and the induction hypothesis we get 
\begin{align*}
\mathcal{N}(P_4, G)&\leq \mathcal{N}(P_4, G-v)+4n-10\leq 2(n-1)^2-7(n-1)+2+(4n-10)=2n^2-7n+1\\&<2n^2-7n+2.
\end{align*}
Notice that in both cases when $G$ contains $2n^2-7n+2$, the graph we get is isomorphic to $K_1+P_{n-1}$ which is shown in Figure~\ref{fig2}.  

To complete proof of the theorem we need to start of the induction step when $n=7$. 

\begin{claim}\label{c8}
$f_{\mathcal{OP}}(7,P_4)=51=2n^{2}-7n+2.$
\end{claim}
\begin{proof}
Let $G$ be a $7$-vertex maximal outerplanar graph. 
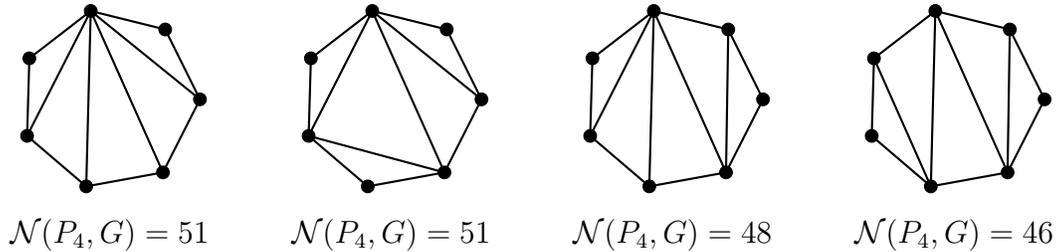
\begin{figure}[h]
\centering
\begin{tikzpicture}[scale=0.06]
\draw[fill=black] (20,0) circle (40pt);
\draw[fill=black] (12.3,15.5) circle (40pt);
\draw[fill=black] (-4.2,19.6) circle (40pt);
\draw[fill=black] (-17.8,9.1) circle (40pt);
\draw[fill=black] (-18.3,-8.1) circle (40pt);
\draw[fill=black] (-5.2,-19.3) circle (40pt);
\draw[fill=black] (11.8,-16.2) circle (40pt);
\draw[thick](20,0)--(12.3,15.5)--(-4.2,19.6)--(-17.8,9.1)--(-18.3,-8.1)--(-5.2,-19.3)--(11.8,-16.2)--(20,0);
\draw[thick](-4.2,19.6)--(-18.3,-8.1)(-4.2,19.6)--(-5.2,-19.3)(-4.2,19.6)--(11.8,-16.2)(-4.2,19.6)--(20,0);
\node at (0,-30) {$\mathcal{N}(P_4,G)=51$};
\end{tikzpicture}\qquad
\begin{tikzpicture}[scale=0.06]
\draw[fill=black] (20,0) circle (40pt);
\draw[fill=black] (12.3,15.5) circle (40pt);
\draw[fill=black] (-4.2,19.6) circle (40pt);
\draw[fill=black] (-17.8,9.1) circle (40pt);
\draw[fill=black] (-18.3,-8.1) circle (40pt);
\draw[fill=black] (-5.2,-19.3) circle (40pt);
\draw[fill=black] (11.8,-16.2) circle (40pt);
\draw[thick](20,0)--(12.3,15.5)--(-4.2,19.6)--(-17.8,9.1)--(-18.3,-8.1)--(-5.2,-19.3)--(11.8,-16.2)--(20,0);
\draw[thick](-4.2,19.6)--(-18.3,-8.1)(-4.2,19.6)--(11.8,-16.2)(-4.2,19.6)--(20,0)(-18.3,-8.1)--(11.8,-16.2);
\node at (0,-30) {$\mathcal{N}(P_4,G)=51$};
\end{tikzpicture}\qquad
\begin{tikzpicture}[scale=0.06]
\draw[fill=black] (20,0) circle (40pt);
\draw[fill=black] (12.3,15.5) circle (40pt);
\draw[fill=black] (-4.2,19.6) circle (40pt);
\draw[fill=black] (-17.8,9.1) circle (40pt);
\draw[fill=black] (-18.3,-8.1) circle (40pt);
\draw[fill=black] (-5.2,-19.3) circle (40pt);
\draw[fill=black] (11.8,-16.2) circle (40pt);
\draw[thick](20,0)--(12.3,15.5)--(-4.2,19.6)--(-17.8,9.1)--(-18.3,-8.1)--(-5.2,-19.3)--(11.8,-16.2)--(20,0);
\draw[thick](-4.2,19.6)--(-18.3,-8.1)(-4.2,19.6)--(-5.2,-19.3)(-4.2,19.6)--(11.8,-16.2)(11.8,-16.2)--(12.3,15.5);
\node at (0,-30) {$\mathcal{N}(P_4,G)=48$};
\end{tikzpicture}\qquad
\begin{tikzpicture}[scale=0.06]
\draw[fill=black] (20,0) circle (40pt);
\draw[fill=black] (12.3,15.5) circle (40pt);
\draw[fill=black] (-4.2,19.6) circle (40pt);
\draw[fill=black] (-17.8,9.1) circle (40pt);
\draw[fill=black] (-18.3,-8.1) circle (40pt);
\draw[fill=black] (-5.2,-19.3) circle (40pt);
\draw[fill=black] (11.8,-16.2) circle (40pt);
\draw[thick](20,0)--(12.3,15.5)--(-4.2,19.6)--(-17.8,9.1)--(-18.3,-8.1)--(-5.2,-19.3)--(11.8,-16.2)--(20,0);
\draw[thick](-4.2,19.6)--(-5.2,-19.3)(-4.2,19.6)--(11.8,-16.2)(-5.2,-19.3)--(-17.8,9.1)(11.8,-16.2)--(12.3,15.5);
\node at (0,-30) {$\mathcal{N}(P_4,G)=46$};
\end{tikzpicture}
\caption{All $7$-vertex maximal outerplanar graphs.}
\end{figure}
It can be checked that there are only four, $7$-vertex maximal outerplanar graphs, see the lists with their number of $3$-paths. There are $2$ graphs attaining the upper bound. This completes the proof of Claim~\ref{c8}.\qedhere  
\end{proof}
This completes the proof of Theorem~\ref{t1}.\qedhere
\end{proof}
\section{Generalized outerplanar Tur\'an number of the $P_{5}$}
\begin{proof}[proof of Theorem~\ref{tms2}]
The following construction verifies the lower bound. 
\begin{definition}
Let $n$ be an even positive integer and $C_n=(v_1,v_2,\dots, v_n)$ be an $n$-vertex cycle. An $n$-vertex maximal outerplanar graph $G_n$(see Figure~\ref{fig2s}) is defined as follows:
\begin{figure}[ht]
\centering
\begin{tikzpicture}[scale=0.2]
\draw[thick](0,20)--(-19.3,-5.2) (0,20)--(-5.2,-19.3)(0,20)--(-10,-17.3)(0,20)--(-17.3,-10)(0,20)--(-14,-14)(0,20)--(-20,0)(0,20)--(-19.3,5.2)(0,20)--(-14,14)(0,20)--(-17.3,10)(0,20)--(-10,17.3);
\draw[thick](0,20)--(10,-17.3)(0,20)--(17.3,-10)(0,20)--(20,0)(0,-20)--(0,20)(0,20)--(17.3,10)(0,20)--(10,17.3);
\draw[black,thick](10,17.3)--(17.3,10)--(20,0)--(17.3,-10)--(10,-17.3)--(0,-20);
\draw[fill=black] (19.3,5.2) circle (15pt);
\draw[fill=black] (-19.3,5.2) circle (15pt);
\draw[fill=black] (-19.3,-5.2) circle (15pt);
\draw[fill=black] (19.3,-5.2) circle (15pt);
\draw[fill=black] (14,14) circle (15pt);
\draw[fill=black] (-14,14) circle (15pt);
\draw[fill=black] (-14,-14) circle (15pt);
\draw[fill=black] (14,-14) circle (15pt);
\draw[fill=black] (5.2,19.3) circle (15pt);
\draw[fill=black] (-5.2,19.3) circle (15pt);
\draw[fill=black] (-5.2,-19.3) circle (15pt);
\draw[fill=black] (5.2,-19.3) circle (15pt);
\draw[thick] (0,0) circle (20cm);
\draw[fill=black] (20,0) circle (15pt);
\draw[fill=black] (-20,0) circle (15pt);
\draw[fill=black] (0,20) circle (15pt);
\draw[fill=black] (0,-20) circle (15pt);
\draw[fill=black] (17.3,10) circle (15pt);
\draw[fill=black] (10,17.3) circle (15pt);
\draw[fill=black] (-17.3,10) circle (15pt);
\draw[fill=black] (-10,17.3) circle (15pt);
\draw[fill=black] (-17.3,-10) circle (15pt);
\draw[fill=black] (-10,-17.3) circle (15pt);
\draw[fill=black] (17.3,-10) circle (15pt);
\draw[fill=black] (10,-17.3) circle (15pt);
\node at (-6,21) {$v_n$};
\node at (-6,-21) {$v_{\frac{n}{2}+1}$};
\node at (6,21) {$v_2$};
\node at (12,18) {$v_3$};
\node at (12,-20) {$v_{\frac{n}{2}-2}$};
\node at (-12,20) {$v_{n-1}$};
\node at (0,23) {$v_1$};
\node at (0,-23) {$v_{\frac{n}{2}}$};
\end{tikzpicture}
\caption{An $n$-vertex maximal outerplanar graph $G_n$.}
\label{fig2s}
\end{figure}
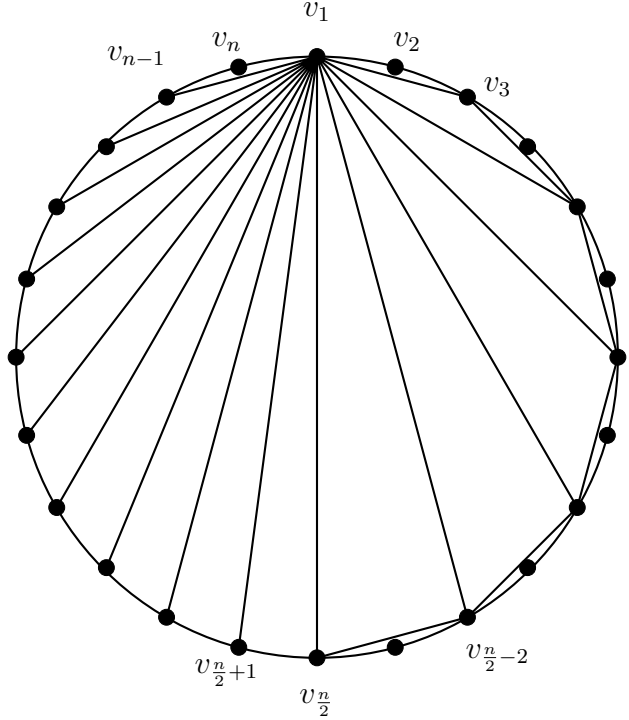
\begin{enumerate}
\item $C_n$ is the outer cycle of $G_n$, 
\item 
\begin{align*}
E(G_n)=&E(C_n)\cup \bigg\{\{v_1,v_i\}: i\in\{\frac{n}{2},\frac{n}{2}+1,\dots,n\}\bigg\}\cup \bigg\{\{v_1,v_i\}: i\in\{3,5,\dots,\frac{n}{2}-2\}\bigg\}\\&\cup \bigg\{\{v_i,v_{i+2}\}: i\in\{3,5,\dots,\frac{n}{2}-2\}\bigg\}    
\end{align*}
\end{enumerate}
\end{definition}

\begin{lemma}\label{lkj}
$\N(P_5,G_n)\approx\frac{17}{4}n^2.$ In other words, $$\lim\limits_{n\longrightarrow\infty}\frac{\N(P_5,G_n)}{\left(\frac{17}{4}n^2\right)}=1.$$
\end{lemma}
\begin{proof}
It is easy to check that the number of $P_5$'s of the form $(u_1,u_2,u_3,u_4,u_5)$, where $u_i\in V(G_n)$ ($i\in[5]$) and $u_j\neq v_1$ (for all $j\in\{2,3,4\}$)  is linear.

Now we count the number of $P_5$'s containing $v_1$ as an interior vertex. We call such a $P_5$ as a \textit{symmetric} $P_5$ if $v_1$ is a middle vertex (see Figure~\ref{fig2ds}, left). Otherwise, we call the $P_5$ as \textit{asymmetric} $P_5$ (see Figure~\ref{fig2ds}, right). Denote these two vertices adjacent to $v_1$ in the symmetric $P_5$ as $x_1$ and $x_2$. Moreover, denote the asymmetric $P_5$ as $(x_1,v_1,x_2,x_3,x_4)$.
\begin{figure}[ht]
\centering
\begin{tikzpicture}[scale=0.15]
\draw[fill=black] (-15,0) circle (15pt);
\draw[fill=black] (-7.5,0) circle (15pt);
\draw[fill=black] (0,10) circle (15pt);
\draw[fill=black] (15,0) circle (15pt);
\draw[fill=black] (7.5,0) circle (15pt);
\draw[thick](-15,0)--(-7.5,0)--(0,10)--(7.5,0)--(15,0);
\node at (0,12) {$v_1$};
\node at (-7.5,-2) {$x_1$};
\node at (7.5,-2) {$x_2$};
\end{tikzpicture}\qquad\qquad
\begin{tikzpicture}[scale=0.15]
\draw[fill=black] (-7.5,0) circle (15pt);
\draw[fill=black] (0,10) circle (15pt);
\draw[fill=black] (15,0) circle (15pt);
\draw[fill=black] (7.5,0) circle (15pt);
\draw[fill=black] (22.5,0) circle (15pt);
\draw[thick](-7.5,0)--(0,10)--(7.5,0)--(15,0)--(22.5,0);
\node at (0,12) {$v_1$};
\node at (-7.5,-2) {$x_1$};
\node at (7.5,-2) {$x_2$};
\node at (15,-2) {$x_3$};
\node at (22.5,-2) {$x_4$};
\end{tikzpicture}
\caption{Structure of $P_5$'s containing $v_1$ as an interior vertex.}
\label{fig2ds}
\end{figure}
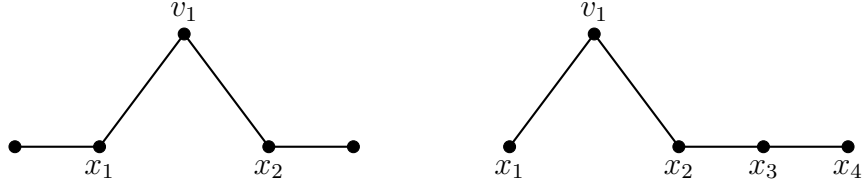

Denote $S_1=\{v_2,v_3,\dots, v_{\frac{n}{2}-1} \}$, $S_2=\{v_{\frac{n}{2}+1},v_{\frac{n}{2}+2},\dots,v_{n-1}\}$. Next we count the number of symmetric and asymmetric $P_5$'s containing $v_1$ as an interior vertex considering the positions of $x_1$ and $x_2$. 
\begin{enumerate}
\item Both $x_1$ and $x_2$ are in $S_1$. 
\begin{itemize}
\item Symmetric $P_5$'s. Since the degree of $x_1$ and $x_2$ are 5, the number of such $P_5$'s is roughly $4\times 4\times {{\frac{n}{4}}\choose 2}=\frac{n^2}{2}.$ 
\item Asymmetric $P_5$'s. It can be checked that the number of $2$-paths of the form $(x_2,x_3,x_4)$ is  8. Since we have two options of doing that, the number of such $P_5$'s is roughly $2\times 8\times {{\frac{n}{4}}\choose 2}=\frac{n^2}{2}.$ 
\end{itemize} 
\item Both $x_1$ and $x_2$ are in $S_2$.
\begin{itemize}
\item Symmetric $P_5$'s. Since the degree of $x_1$ and $x_2$ is 3, the number of such $P_5$'s is roughly $2\times 2\times {{\frac{n}{2}}\choose 2}=\frac{n^2}{2}.$ 
\item Asymmetric $P_5$'s. Clearly the number of $2$-paths of the form $(x_2,x_3,x_4)$ is 2. Since we have two options of doing that, the number of such $P_5$'s is roughly $2\times 2\times {{\frac{n}{2}}\choose 2}=\frac{n^2}{2}.$ 
\end{itemize} 
\item $x_1\in S_1$  and $x_2\in S_2$ or vice versa.
\begin{itemize}
\item Symmetric $P_5$'s. Without loss of generality, we assume that $x_1\in S_1$ and $x_2\in S_2$. Hence, the degree of $x_1$ and $x_2$ are respectively 3 and 5. Thus, the number of symmetric $P_5$'s is roughly $2\times 4\times\frac{n}{2}\times \frac{n}{4}=n^2.$ 
\item Asymmetric $P_5$'s. For the case when $x_2$ is in $S_1$, the number of $2$-paths of the form $(x_2,x_3,x_4)$ is 8. Thus the number of such asymmetric $P_5$'s is roughly $8\times \frac{n}{2}\times \frac{n}{4}=n^2.$ 

On the other hand for the case that $x_2$ is in $S_2$, the number of such asymmetric $P_5$'s is $2\times \frac{n}{2}\times \frac{n}{4}=\frac{n^2}{4}.$ 
\end{itemize} 
\end{enumerate}
Therefor, $\N(P_5,G_n)\approx\frac{17}{4}n^2.$ This completes proof of Lemma~\ref{lkj}.\qedhere
\end{proof}

The following lemma claims that no mater how these $\frac{n}{4}$ degree $2$ vertices are distributed in the outer cycle $C_n$, the maximal outerplanar graph contains roughly $\frac{17}{4}n^2$ $P_5$'s.
\begin{lemma}\label{gen}
Let $n$ be even and $G$ be an $n$-vertex maximal outerplanar with outer cycle $C_n$ and let $G$ contains  $\frac{n}{4}$ degree $2$ vertices and all the remaining vertices are adjacent to a vertex, say $v\in V(C_n)$. Then $$\N(P_5,G)\approx\frac{17}{4}n^2.$$  
\end{lemma}
\begin{proof}
Denote vertices which are adjacent to $v$ in counterclockwise direction as $v_1,v_2,\dots, v_r$. Obviously $r=\frac{3n}{4}+3$. Denote a degree $2$ vertex (if exist) between $v_i$ and $v_{i+1}$ in the counterclockwise direction by $v_i^1$. Notice that in such case $\{v_i,v_{i+1}\}\in E(G)$. 

Next we count the number of $P_5$'s containing $v$ (but not as a terminal vertex). Precisely speaking we count the number of symmetric and asymmetric $P_5$'s shown in Figure~\ref{fig2dss}.
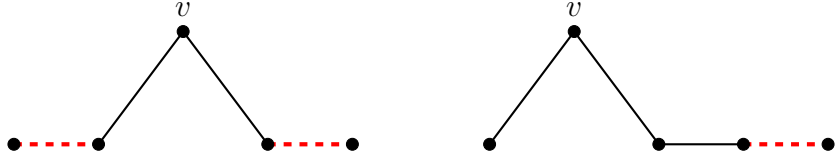
\begin{figure}[ht]
\centering
\begin{tikzpicture}[scale=0.15]
\draw[thick](-7.5,0)--(0,10)--(7.5,0);
\draw[dashed,red,ultra thick](-15,0)--(-7.5,0)(7.5,0)--(15,0);
\draw[fill=black] (-15,0) circle (15pt);
\draw[fill=black] (-7.5,0) circle (15pt);
\draw[fill=black] (0,10) circle (15pt);
\draw[fill=black] (15,0) circle (15pt);
\draw[fill=black] (7.5,0) circle (15pt);
\node at (0,12) {$v$};
\end{tikzpicture}\qquad\qquad
\begin{tikzpicture}[scale=0.15]
\draw[thick](-7.5,0)--(0,10)--(7.5,0)--(15,0);
\draw[dashed,red,ultra thick](15,0)--(22.5,0);
\draw[fill=black] (-7.5,0) circle (15pt);
\draw[fill=black] (0,10) circle (15pt);
\draw[fill=black] (15,0) circle (15pt);
\draw[fill=black] (7.5,0) circle (15pt);
\draw[fill=black] (22.5,0) circle (15pt);
\node at (0,12) {$v$};
\end{tikzpicture}
\caption{$P_5$'s containing $v$ as an interior vertex.}
\label{fig2dss}
\end{figure}
\begin{itemize}
\item Asymmetric $P_5$'s. Here we count in how many ways each edge, which is different from $\{v,v_{i}\}$ $i\in[k]$, is contained in an asymmetric $P_5$ as terminal edge (the red edge which is shown in Figure~\ref{fig2dss}, right). 

For an edge $\{v_i,v_{i+1}\}$, it can be checked that it is contained in two possibilities for each other vertex $v_j\notin\{v_{i-1},v_i,v_{i+1},v_{i+2}\}$. Indeed, $(v_j, v, v_{i-1},v_i,v_{i+1})$ and $(v_j, v, v_{i+2},v_{i+1},v_i)$. The number of such edges is roughly $2n-\frac{3n}{4}-2\times\frac{n}{4}=\frac{3n}{4}.$ This implies, the number of $P_5$'s containing the edges $\{v_i,v_{i+1}\}$ as terminal edge is $\frac{3n}{4}\times 2\times \frac{3n}{4}=\frac{9n^2}{8}.$ 

Let $v_i^1$ be a degree 2 vertex between $v_i$ and $v_{i+1}$. Now we count in how many ways the edge $\{v_i,v_i^1\}$ is possibly contained in an asymmetric $P_5$ as a terminal edge. In this case for each $v_j\notin\{v_{i-1},v_i,v_{i+1}\}$, it can be checked that the edge $\{v_i,v_i^1\}$ is contained in three such $P_5$'s. Indeed, the paths are $(v_j,v,v_{i-1},v_i,v_i^1), (v_j, v,v_{i+1},v_i^1,v_i)$ and $(v_j,v,v_{i+1},v_i,v_{i}^1)$. Notice that the same property hold for the edge $\{v_i^1,v_{i+1}\}$. Clearly the number of such edges id $2\times \frac{n}{4}$. Therefore, the number of $P_5$'s containing the edge $\{v_j,u\}$, where $u$ is a degree two vertex, as a terminal edge is $\frac{3n}{4}\times 3\times \frac{n}{2}=\frac{9n^2}{8}$.
\item Symmetric $P_5$'s. Here $P_5$ is of the form shown in Figure~\ref{fig2dss}(left). In counting such $P_5$'s we consider three possibilities considering the two terminal red edges. Let the red edges are of the form $\{v_i,v_{i+1}\}$ and $\{v_j,v_{j+1}\}$. In this case we have there are four $P_5$'s of such a form. Indeed the paths are $(v_{i+1},v_i,v,v_{j+1},v_j), (v_{i+1},v_i,v,v_{j},v_{j+1}), \\ (v_{i},v_{i+1},v,v_{j+1},v_j) $ and $(v_{i},v_{i+1},v,v_{j},v_{j+1})$. Notice that we have roughly $\frac{3n}{4}\choose 2$ such pairs of edges. Therefore the number of symmetric $P_5$'s such that the terminal red edges are both in $\{v_i,v_{i+1}\}$ is roughly $4\times {\frac{3n}{4}\choose 2}\approx\frac{9n^2}{8}$. 

It is easy to check that number of symmetric $P_5$ such that both terminal red edges are in $\{u,v_i\}$, where $u$ is a degree two vertex, is only one. Roughly such pair of edges is ${\frac{n}{2}\choose 2}\approx \frac{n^2}{8}.$ 

Finally we count the the number of symmetric $P_5$'s such that one terminal red edge is of the form $\{v_i,v_{i+1}\}$ and the other terminal red edge is of the form $\{v_j,v_j^1\}$. Here it can be checked that the number of such $P_5$'s is two, namely, $(v_i,v_{i+1},v,v_{j},v_{j}^1)$ and $(v_{i+1},v_i,v,v_{j},v_{j}^1)$. Clearly the number of such pair of edges is roughly $\frac{3n}{3}\times \frac{n}{2}=\frac{3n^2}{8}.$ Thus, the number of $P_5$'s of such kind is roughly $2\times \frac{3n^2}{8}=\frac{6n^2}{8}.$    
\end{itemize}
Therefor, adding all the symmetric and asymmetric $P_5$'s we have, $$\N(P_5,G)\approx \frac{9n^2}{8}+\frac{9n^2}{8}+\frac{9n^2}{8}+\frac{n^2}{8}+\frac{6n^2}{8}=\frac{17}{4}n^2.$$
This completes proof of Lemma~\ref{gen}.\qedhere
\end{proof}

We need the following lemmas to complete the proof of Theorem~\ref{tms2}.
\begin{lemma}\label{ooo}
Let $T$ be an $n$-vertex tree with $\Delta(T)\leq 3$. Then there is an edge $e\in E(T)$ such that both components of $T-e$ have at least $\frac{n-1}{3}$ vertices. 
\end{lemma}
\begin{proof}
We may assume that $\Delta(T)\geq 2$. Let $k$ be the number of degree 3 vertices in $T$. Clearly the lemma holds when $k=0$. Indeed, if $k=0$, then $T$ is a path. We can take an edge $e$ which is incident to a center vertex of $T$ so that both components of $T-e$ have at least $\frac{n-1}{2}>\frac{n-1}{3}$ vertices.  

Next we apply induction on the number of degree 3 vertices in $T$. Let $k\geq 1$, $v$ be a degree 3 vertex in $T$ and $N(v)=\{v_1,v_2,v_3\}$. For $i\in[3]$, denote the maximal sub-tree of $T$ containing the vertex $v_i$ but not $v$ by $T_i$. Let $v(T_i)=n_i$, $i\in[3]$. Clearly either $n_1$ or $n_2$ or $n_3$ is at least $\frac{n-1}{3}$. 
Without loss of generality assume that $n_1\geq \frac{n-1}{3}$. If $n_2+n_3\geq \frac{n-1}{3}$, we are done with the choice of $e=\{v,v_1\}$. 

Let $n_2+n_3<\frac{n-1}{3}$. Consider the tree $T'$ which is obtained from $T$ by deleting the vertices in $T_2$ and $T_3$ and identifying the vertex $v$ and a terminal vertex (say $r_1$) of an $(n_2+n_3+1)$-vertex path $(r_1,r_2,\dots,r_{n_2+n_3+1})$. Obviously, $v(T')=n$ and the number of degree-$3$ vertices in $T'$ is less than $k$. So by induction, there is an edge $e\in E(T')$ such that both components $T'-e$ has at least $\frac{n-1}{3}$ vertices. For clear reason $e$ can not be an edge of the path $(v_1,v,r_2,\dots,r_{n_2+n_3+1})$. In other words, $e\in E(T_1)$. Therefore with the choice of edge $e$, both components of $T-e$ has at least $\frac{n-1}{3}$ vertices. 

This completes the proof of Lemma~\ref{ooo}.\qedhere
\end{proof}

\begin{definition}
Let $G$ be a maximal outerplanar graph and $C$ be its outer cycle. An edge $\{x,y\}\in E(G)$ is called a chord if $x$ and $y$ are two nonconsecutive vertices of $C$.
\end{definition}
\begin{lemma}\label{oom}
Let $G$ be an $n$-vertex maximal outerplanar graph. Then there is a chord $\{x,y\}$ in $G$ such that the number of interior vertices between $x$ and $y$ in counterclockwise and clockwise directions are greater than $\frac{n}{3}$. 
\end{lemma}
\begin{proof}
Let the outer cycle of $G$ be $C$. Each edge in $E(G)\backslash E(C)$ are contained in two triangular face. Let $T$ be the dual of $G$ minus the vertex representing the unbounded face. Clearly $T$ is a tree, with $\Delta(T)\leq 3$. Since $G$ contains $n-2$ triangles, $v(T)=n-2$.  

By Lemma~\ref{ooo}, there is an edge $e\in E(T)$ such that $T-e$ results components with at least $\frac{n-3}{3}$ vertices. Let the end vertices of $e$ be $x$ and $y$ and $e'$ be an edge in $G$ which is shared by the triangular faces represented by $x$ and $y$ respectively in the dual of $G$. Thus the number of vertices between the end vertices of $e'$ in counterclockwise and clockwise directions are at least $\frac{n-3}{3}+2>\frac{n}{3}$. This completes the proof of Lemma~\ref{oom}.\qedhere
\end{proof}
\begin{definition}
Let $G$ be a maximal outerplanar graph. A chord $\{x,y\}$ in $G$ is called a \say{nice chord} if the number of vertices between $x$ and $y$ (including $x$ and $y$) in counterclockwise and clockwise direction are at least $\frac{n}{3}.$  
\end{definition}
\begin{definition}
Let $G$ be a maximal outerplanar graph and $\{x,y\}$ be a chord in $G$. Let $I_1$ be a set of interior vertices between $x$ and $y$ in counterclockwise direction and $I_2$ be the set of interior vertices between $x$ and $y$ in clockwise direction. A path $P$ of length four in $G$ is called a crossing $P_5$ with respect to $\{x,y\}$ if $P$ contains a vertex in both $I_1$ and $I_2$. 
\end{definition}
Notice that every crossing $P_5$ with respect to the chord $\{x,y\}$ contains either $x$ or $y$.  We classify the crossing $P_5$'s with respect to $\{x,y\}$ as follows.
\begin{enumerate}
\item[i.] A crossing $P_5$ with respect to $\{x,y\}$ is called \textbf{Type-I crossing $P_5$} if it contains the edge $\{x,y\}$. Notice that, the edge $\{x,y\}$ can not be a terminal edge of a crossing $P_5$. Moreover, such a crossing $P_5$ contains only one vertex in $I_1$ and two vertices in $I_2$ or vice-versa. See Figure~\ref{fi2}.
\begin{figure}[ht]
\centering
\begin{tikzpicture}[scale=0.06]
\draw[dashed,red, ultra thick](0,-20)--(0,20)(-17.3,-10)--(0,-20)(0,20)--(17.3,10)--(17.3,-10);
\draw[thick] (0,0) circle (20cm);
\draw[fill=black](0,20)circle(40pt);
\draw[fill=black](0,-20)circle(40pt);
\draw[fill=black] (17.3,10) circle (40pt);
\draw[fill=black] (-17.3,-10) circle (40pt);
\draw[fill=black] (17.3,-10) circle (40pt);
\node at (0,-25) {$y$};
\node at (0,25) {$x$};
\end{tikzpicture}
\caption{An example of Type-I crossing $P_5$ with respect to $\{x,y\}$.}
\label{fi2}
\end{figure}
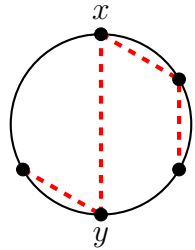
\item[ii.] A crossing $P_5$ with respect to $\{x,y\}$ is called \textbf{Type-II crossing $P_5$} if it contains both $x$ and $y$, but not the edge $\{x,y\}$. Notice that $x$ and $y$ can not be both terminal vertices. In this case we have two possibilities based on the position of $x$ and $y$.
\begin{itemize}
\item If both $x$ and $y$ are not terminal vertex of the $P_5$, then we call the crossing $P_5$ as \textbf{Type-II(A) crossing $P_5$}, see Figure~\ref{fi22} (the first two graphs).
\begin{figure}[ht]
\centering
\begin{tikzpicture}[scale=0.06]
\draw[ultra thick](0,-20)--(0,20);
\draw[dashed,red, ultra thick](-17.3,10)--(0,20)--(20,0)--(0,-20)--(-17.3,-10);
\draw[thick] (0,0) circle (20cm);
\draw[fill=black](0,20)circle(40pt);
\draw[fill=black](0,-20)circle(40pt);
\draw[fill=black] (-17.3,10) circle (40pt);
\draw[fill=black] (-17.3,-10) circle (40pt);
\draw[fill=black] (20,0) circle (40pt);
\node at (0,-25) {$y$};
\node at (0,25) {$x$};
\end{tikzpicture}\qquad\qquad
\begin{tikzpicture}[scale=0.06]
\draw[ultra thick](0,-20)--(0,20);
\draw[dashed,red, ultra thick](-17.3,10)--(0,20)--(20,0)--(0,-20)--(17.3,-10);
\draw[thick] (0,0) circle (20cm);
\draw[fill=black](0,20)circle(40pt);
\draw[fill=black](0,-20)circle(40pt);
\draw[fill=black] (-17.3,10) circle (40pt);
\draw[fill=black] (17.3,-10) circle (40pt);
\draw[fill=black] (20,0) circle (40pt);
\node at (0,-25) {$y$};
\node at (0,25) {$x$};
\end{tikzpicture}\qquad\qquad
\begin{tikzpicture}[scale=0.06]
\draw[ultra thick](0,-20)--(0,20);
\draw[dashed,red, ultra thick](-17.3,-10)--(0,-20)--(17.3,-10)--(17.3,10)--(0,20);
\draw[thick] (0,0) circle (20cm);
\draw[fill=black](0,20)circle(40pt);
\draw[fill=black](0,-20)circle(40pt);
\draw[fill=black] (17.3,10) circle (40pt);
\draw[fill=black] (17.3,-10) circle (40pt);
\draw[fill=black] (-17.3,-10) circle (40pt);
\node at (0,-25) {$y$};
\node at (0,25) {$x$};
\end{tikzpicture}\qquad\qquad
\begin{tikzpicture}[scale=0.06]
\draw[ultra thick](0,-20)--(0,20);
\draw[dashed,red, ultra thick](-17.3,10)--(-17.3,-10)--(0,-20)--(20,0)--(0,20);
\draw[thick] (0,0) circle (20cm);
\draw[fill=black](0,20)circle(40pt);
\draw[fill=black](0,-20)circle(40pt);
\draw[fill=black] (20,0) circle (40pt);
\draw[fill=black] (-17.3,10) circle (40pt);
\draw[fill=black] (-17.3,-10) circle (40pt);
\node at (0,-25) {$y$};
\node at (0,25) {$x$};
\end{tikzpicture}
\caption{Examples of Type-II crossing $P_5$ with respect to $\{x,y\}$.}
\label{fi22}
\end{figure}
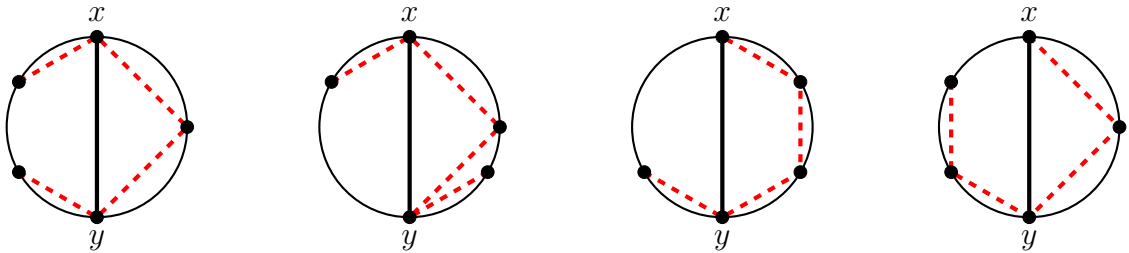
\item If either $x$ or $y$ is a terminal vertex of the $P_5$, then we call the $P_5$ as \textbf{Type-II(B) crossing $P_5$}, see Figure~\ref{fi22} (last two graphs).
\end{itemize}
\item [iii.]A crossing $P_5$ with respect to $\{x,y\}$ is called \textbf{Type-III crossing $P_5$} if the crossing $P_5$ contains either $x$ or $y$ but not both.
\begin{itemize}
\item If the crossing $P_5$ contains two vertices in $I_1$ and two vertices of $I_2$, then we call the $P_5$ as \textbf{Type-III(A) crossing $P_5$}, see Figure~\ref{fi11}(left).

\item If the crossing $P_5$ contains only one vertex in $I_1$ or $I_2$, then we call the $P_5$ as \textbf{Type-III(B) crossing $P_5$}, see Figure~\ref{fi11}(right).
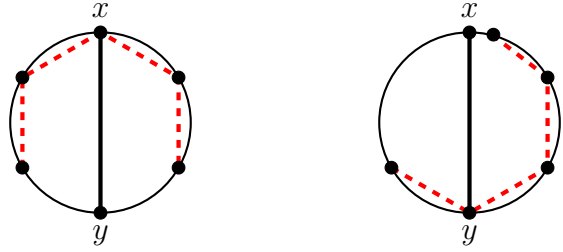
\begin{figure}[ht]
\centering
\begin{tikzpicture}[scale=0.06]
\draw[ultra thick](0,-20)--(0,20);
\draw[dashed,red, ultra thick](-17.3,-10)--(-17.3,10)--(0,20)--(17.3,10)--(17.3,-10);
\draw[thick] (0,0) circle (20cm);
\draw[fill=black](0,20)circle(40pt);
\draw[fill=black](0,-20)circle(40pt);
\draw[fill=black] (17.3,10) circle (40pt);
\draw[fill=black] (-17.3,10) circle (40pt);
\draw[fill=black] (-17.3,-10) circle (40pt);
\draw[fill=black] (17.3,-10) circle (40pt);
\node at (0,-25) {$y$};
\node at (0,25) {$x$};
\end{tikzpicture}\qquad\qquad\qquad
\begin{tikzpicture}[scale=0.06]
\draw[ultra thick](0,-20)--(0,20);
\draw[dashed,red, ultra thick](-17.3,-10)--(0,-20)--(17.3,-10)--(17.3,10)--(5.2,19.3);
\draw[thick] (0,0) circle (20cm);
\draw[fill=black](0,20)circle(40pt);
\draw[fill=black](0,-20)circle(40pt);
\draw[fill=black] (17.3,10) circle (40pt);
\draw[fill=black] (5.3,19.5) circle (40pt);
\draw[fill=black] (-17.3,-10) circle (40pt);
\draw[fill=black] (17.3,-10) circle (40pt);
\node at (0,-25) {$y$};
\node at (0,25) {$x$};
\end{tikzpicture}
\caption{Examples of Type-III crossing $P_5$ with respect to $\{x,y\}$.}
\label{fi11}
\end{figure}
\end{itemize}
\end{enumerate}

\begin{definition}
Let $G$ be a maximal outerplanar graph and $e$ be a chord in $G$. Denote the number of crossing $P_5$'s with respect to $e$ by $\mathcal{N}_{e}(P_5,G)$.
\end{definition}

\begin{lemma}\label{mainlema}
Let $G$ be an $n$-vertex maximal outerplanar graph and $e=\{x,y\}$ be a nice chord in $G$. Let $n_1$ be the number of interior vertices laying between $x$ and $y$ in counterclockwise direction and $n_2$ be the number of interior vertices between $x$ and $y$ in clockwise direction. Then $$\mathcal{N}_{e}(P_5,G)=\frac{17}{2}n_1n_2+O(n).$$
\end{lemma}
\begin{proof}
Let $C$ be the outer cycle of $G$ and $I_1$ and $I_2$ be the set of interior vertices between $x$ and $y$ in counterclockwise and clockwise direction respectively. Let $x$ be adjacent with $s_1$ vertices in $I_1$ and the vertices in counterclockwise direction be $\tilde{x}_1,\tilde{x}_2,\dots,\tilde{x}_{s_1}$. Let $y$ be adjacent with $t_1$ vertices in $I_1$ and the vertices in clockwise direction be $\tilde{y}_1,\tilde{y}_2,\dots,\tilde{y}_{t_1}$. Obviously $\tilde{y}_{t_1}$ and $\tilde{x}_{s_1}$ are identical vertices. Similarly let $x$ and $y$ be adjacent with $p_1$ and $q_1$ number of vertices in $I_2$ respectively. Let the vertices which are adjacent to $x$ in clockwise direction and adjacent to $y$ in counterclockwise direction respectively be $\hat{x}_1,\hat{x}_2,\dots,\hat{x}_{p_1}$ and $\hat{y}_1,\hat{y}_2,\dots,\hat{y}_{q_1}$. Obviously, $\hat{x}_{p_1}$ and $\hat{y}_{q_1}$ are identical vertices. The remaining vertices of $G$ are labeled as follows.
\begin{itemize}
\item Label the interior vertices in clockwise direction (if any) between $\hat{x}_i$ and $\hat{x}_{i+1}$ by $\hat{x}_i^1, \hat{x}_i^2,\dots \hat{x}_{i}^{n_i}$, $i\in[p_1-1]$.
\item Label the interior vertices between $\hat{y}_{j}$ and $\hat{y}_{j+1}$ in counterclockwise direction (if any) by $\hat{y}_{j}^1,\hat{y}_{j}^2,\dots, \hat{y}_{j}^{n_{j}}$, $j\in[q_1-1]$.
\item  Label the interior vertices in counterclockwise direction (if any) between $\tilde{x}_k$ and $\tilde{x}_{k+1}$ by $\tilde{x}_{k}^1,\tilde{x}_k^2,\dots,\tilde{x}_k^{n_k}$, $k\in[s_1-1]$.
\item Label the interior vertices in clockwise direction (if any) between $\tilde{y}_\ell$ and $\tilde{y}_{\ell+1}$ by $\tilde{y}_{\ell}^1,\tilde{y}_\ell^2,\dots,\tilde{y}_\ell^{n_\ell}$, $\ell\in[t_1-1]$. 
\end{itemize}

Since each face (except the unbounded face) of $G$ are triangular, for each $i\in[p_1-1], j\in[q_1-1], k\in[s_1-1]$ and $\ell\in[t_1-1]$, the edges $\{\hat{x}_{i},\hat{x}_{i+1}\}$, $\{\hat{y}_j,\hat{y}_{j+1}\}$, $\{\tilde{x}_{k},\tilde{x}_{k+1}\}$ and $\{\tilde{y}_{\ell},\tilde{y}_{\ell+1}\}$ are in $G$. The following claim is important to finish proof of the lemma.

\begin{claim}\label{cv1}
If a maximal outerplanar graph $H$ has the property that either there is no or only one interior vertex between each pair $(\tilde{x}_{k}, \tilde{x}_{k+1})$ counterclockwise direction, $(\tilde{y}_{\ell}, \tilde{y}_{\ell+1})$ clockwise direction, $(\hat{x}_{i},\hat{x}_{i+1})$ clockwise direction and $(\hat{y}_{j},\hat{y}_{j+1})$ counterclockwise direction, where $i\in[p_1-1], j\in[q_1-1], k\in[s_1-1]$ and $\ell\in[t_1-1]$, then $$\mathcal{N}_{e}(P_5,H)=\frac{17}{2}n_1n_2+O(n).$$    

\end{claim}
\begin{proof} 
Let the set of interior vertices between $x$ and $\tilde{x}_{s_1}$ in counter clockwise direction be $S$. Denote the size of $S$ with $s$ and the number of degree 2 vertices in $S$ by $s_0$. Denote the set of vertices between $y$ and $\tilde{y}_{t_1}$ in clockwise direction by $T$. Denote the size of $T$ and degree 2 vertices in $T$ by $t$ and $t_0$ respectively. Denote the set of vertices between $x$ and $\hat{x}_{p_1}$ in clockwise direction by $P$. Moreover denote the size of $P$ and degree 2 vertices in $P$ by $p$ and $p_0$  respectively. Finally, denote the set of vertices between $y$ and $\hat{y}_{q_1}$ in counterclockwise direction by $Q$. Moreover denote the size of $Q$ and degree 2 vertices in $Q$ by $q$ and $q_0$  respectively. 

Obviously $s=s_0+s_1, t=t_0+t_1, p=p_0+p_1$ and $q=q_0+q_1$. Moreover, $n_1=s+t-1$ and $n_2=p+q-1$. Next we count the number of crossing $P_5$'s with respect to the chord $e$, and we do this type-wise.

\begin{enumerate}
\item Type-I crossing $P_5$'s with respect to $e$: First we count the number of crossing $P_5$'s of the form $(\tilde{x}_{k},x,y,\hat{y}_{j},v)$. It can be checked that for a vertex $v\in Q$, the number of cherries with terminal vertices $y$ and $v$ is at most $2$. So, the number of Type-I crossing $P_5$'s of the form $(\tilde{x}_{k},x,y,\hat{y}_{j},v)$ is at most $2s_1q$. Considering the other possibilities, the number of Type-I crossing $P_5$'s is at most $2(s_1q+q_1s+t_1p+p_1t)$, which is at most $4sq+4tp$.

\item Type-II(A) crossing $P_5$'s with respect to $e$: It can be checked that the number of Type-II(A)
crossing $P_5$'s with respect to $e$ is at most $s_1t_1+p_1q_1+2(s_1q_1+t_1p_1)$, which is at most $st+pq+2(sq+tp)$.
\item Type-II(B) crossing $P_5$'s with respect to $e$: We ignore this kind of crossing $P_5$'s as the number is linear.
\item Type-III(A) crossing $P_5$'s with respect to $e$: First we count number of crossing $P_5$'s of the form $(v_1,\tilde{x}_{k},x,\hat{x}_{i},v_1)$. Clearly for a vertex $v_1\in S$  the number of cherries with terminal vertices $v_1$ and $x$ is at most $2$. Moreover, for a vertex $v_2\in P$, the number of cherries with terminal vertices $x$ and $v_2$ is at most $2$. Therefore, number of Type-III(A) crossing $P_5$ of the form $(v_1,\tilde{x}_{k},x,\hat{x}_{i},v_1)$ is $4sp$. Hence considering the other possibility, the number of Type-III(A) crossing $P_5$'s in $G$ is $4sp+4tq$.

\item Type-III(B) crossing $P_5$'s with respect to $e$: First we count the number of Type-III(B) crossing $P_5$'s of the form $(\tilde{x}_{k},x,\hat{x}_{i},v_1,v_2)$, where $k\in[s_1]$, $i\in[p_1]$ and $v_1,v_2\in P$. Considering the structure of $G$ of the claim, the edge $\{v_1,v_2\}$ is either $\{\hat{x}_{i},\hat{x}_{i+1}\}$ or $\{v,\hat{x}_{i}\}$ or $\{v,\hat{x}_{i+1}\}$, for some $i\in[p_1-1]$ and $v$ a degree 2 vertex in $P$. For each edge $\{\hat{x}_{i},\hat{x}_{i+1}\}$, $i\in[p_1-1]$, the number of $3$-paths of the form $(x,h,\hat{x}_{i},\hat{x}_{i+1})$ or $(x,h,\hat{x}_{i+1},\hat{x}_{i})$ is 2, namely $(x,\hat{x}_{i+2},\hat{x}_{i+1},\hat{x}_{i})$ and $(x,\hat{x}_{i-1},\hat{x}_{i},\hat{x}_{i+1})$. Let $v\in P$ be a degree 2 vertex be an interior vertex between $\hat{x}_{i}$ and $\hat{x}_{i+1}$ in clockwise direction. Notice that we have $p_1$ number of edges of the form $\{\hat{x}_i,\hat{x}_{i+1}\}$. Next we count the number of $3$-paths of the form $(x,h,\hat{x}_{i},v)$ and $(x,h,v,\hat{x}_{i})$ and similar argument can be given for the number of $3$-paths of the form $(x,h,\hat{x}_{i+1},v)$ and $(x,h,v,\hat{x}_{i+1})$. There are only three $3$-paths of the form $(x,h,\hat{x}_{i},v)$ and $(x,h,v,\hat{x}_{i})$, namely $(x,\hat{x}_{i-1},\hat{x}_{i},v)$, $(x,\hat{x}_{i+1},v,\hat{x}_{i})$ and $(x,\hat{x}_{i+1},\hat{x}_{i},v)$. Notice that for each degree 2 vertex $v\in P$, we have 2 edges this kind, namely $\{v,\hat{x}_{i}\}$ and $\{v,\hat{x}_{i+1}\}$.

Thus, the number of $3$-paths of the form $(x,\hat{x}_{i},v_1,v_2)$ , where $i\in[p_1]$ and $v_1,v_2\in P$ is at most $2p_1+6p_0$. Therefore, the number of Type-III(B) crossing $P_5$'s of the form $(\tilde{x}_{k},x,\hat{x}_{i},v_1,v_2)$, where $k\in[s_1]$, $i\in[p_1]$ and $v_1,v_2\in P$ is at most $s_1(2p_1+6p_0)$.

With similarly, the number of Type-III(B) crossing $P_5$'s of the form $(\hat{x}_{i},x,\hat{x}_{k},v_1,v_2)$, where $k\in[s_1]$, $i\in[p_1]$ and $v_1,v_2\in S$ is at most $p_1(2s_1+6s_0)$.

Thus taking the sum of the two and $p_0=p-p_1$ and $s_0=s-s_1$ the number of Type-III(B) crossing $P_5$'s with respect to $e$ and containing $x$ is at most $4s_1p_1+6s_1(p-p_1)+6p_1(s-s_1)$.

Notice that $p_1\in[1,p]$ and $s_1\in[1,s]$. Now define a function $g:[1,s]\times[1,p]\longrightarrow\mathbb{R}$ as  $g(s_1,p_1):=4s_1p_1+6s_1(p-p_1)+6p_1(s-s_1)$. It can be checked that $g$ attains maximum value at the critical point $(\frac{3}{4}s,\frac{3}{4}p)$. Moreover at this point the value of $g$ is $g(\frac{3s}{4},\frac{3p}{4})=\frac{9}{2}ps.$ That means the number of Type-III(B) crossing $P_5$'s containing the vertex $x$ is at most $\frac{9ps}{2}$. 

With similar argument, the number of Type-III(B) crossing $P_5$'s containing $y$ is at most $\frac{9}{2}tq$. Therefore, the number of Type-III(B) crossing $P_5$'s with respect to $e$ is at most $\frac{9}{2}ps+\frac{9}{2}tq$.   
\end{enumerate}
Summing up all the upper bound, the number of crossing $P_5$'s with respect to $e$ is at most 
$(4sq+4tp)+(st+pq+2sq+2tp)+(4sp+4tq)+(\frac{9}{2}ps+\frac{9}{2}tq)$, which is $8.5ps+8.5tq+6sq+6tp+st+pq.$ Considering that $q=n_1-p$ and $t=n_2-s$, the last expression is equal to $8.5n_1n_2+5sp+n_1s+n_2q-s^2-p^2-\frac{5}{2}n_2s-\frac{5}{2}n_1p$. Notice that $s\in[1,n_1]$ and $p\in[1,n_2]$. Define a function $h(s,p):[1,n_1] \times [1,n_2]\longrightarrow\mathbb{R}$ as $h(s,p):=8.5n_1n_2+5sp+n_1s+n_2q-s^2-p^2-\frac{5}{2}n_2s-\frac{5}{2}n_1p.$ We treat $h$ as a continuous real valued function in the rectangular region to find the maximum value of $h$. 

It can be checked that $h$ possesses a critical point $(\frac{1}{2}n_1,\frac{1}{2}n_2)$. The values of $h$ at the points $(1,1),(1,n_2),(n_1,1),(n_1,n_2)$ and $(\frac{1}{2}n_1,\frac{1}{2}n_2)$ are respectively, $f(1,1)=\frac{17}{2}n_1n_2-\frac{3}{2}n_1-\frac{3}{2}n_2+3$, $h(1,n_2)=6n_1n_2+n_2+\frac{5}{2}n_2-1$, $h(n_1,1)=n_1n_2+n_2+\frac{5}{2}n_1-1$, $h(n_1,n_2)=\frac{17}{2}n_1n_2$ and $h(\frac{1}{2}n_1,\frac{1}{2}n_2)=\frac{29}{4}n_1n_2+\frac{1}{4}n_1^2+\frac{1}{4}n_2^2=\frac{17}{2}n_1n_2+\frac{1}{4}n_1^2+\frac{1}{4}n_2^2-\frac{5}{4}n_1n_2$. 

To finish our proof we need to show that $h(\frac{1}{2}n_1,\frac{1}{2}n_2)=\frac{17}{2}n_1n_2+O(n)$. Indeed, remember that $n=n_1+n_2$ and $\frac{n}{3}\leq n_1\leq \frac{2}{3}n$. Define a function $f(n_1):[\frac{1}{3}n,\frac{2}{3}n]\longrightarrow\mathbb{R}$ as $f(n_1)=\frac{1}{4}n_1^2+\frac{1}{4}(n-n_1)^2-\frac{5}{4}n_1(n-n_1)$. It is easy to check that $f$ attains maximum value when $n_1=\frac{n}{2}$. Moreover the maximum value of $f$ is $f(\frac{1}{2}n)=-\frac{3}{16}n^2$. 
This completes the proof of Claim~\ref{cv1}.\qedhere
\end{proof}

From Claim~\ref{cv1}, without loss of generality we may assume that $i^*\in [p_1-1]$ is the smallest integer such that there is at least two interior vertices in clockwise direction between $\hat{x}_{i^*}$ and $\hat{x}_{i^*+1}$. Hence, there is no or one interior vertex in clockwise direction between $\hat{x}_i$ and $\hat{x}_{i+1}$, for each $i\in[i^*-1]$. Notice that $\{\hat{x}_{i^*}, \hat{x}_{i^*+1}\}$ is an edge in $G$.

Let $I^*$ be the set of interior vertices between $\hat{x}_{i^*}^1$ and $\hat{x}_{i^*+1}$(clockwise direction) which are adjacent to $\hat{x}_{i^*}$. Denote $I=I^*\cup\{\hat{x}_{i^*+1}\}$.  

Let $G'$ be the maximal outerplanar graph obtained from $G$ with three consecutive operations. Call the operations collectively as a \textit{reduction operation}.
\begin{enumerate}
\item Delete the edges $\{\hat{x}_{i^*}, w\}$, for each $w\in I$.
\item Add the edges $\{\hat{x}_{i^*}^1 ,w\}$,  for each $w\in I$. Where $\hat{x}_{i^*}^1$ is the first interior vertex in clockwise direction between $\hat{x}_{i^*}$ and $\hat{x}_{i^*+1}$. Notice that there is exactly one vertex in $I$ which is adjacent to $\hat{x}_{i^*}^1.$ Denote the vertex by $w^*$. Notice that $w^*$ can be $\hat{x}_{i^*+1}$.
\item Delete the multiple edge which results from procedure (2) and finally add the edge $\{x,\hat{x}_{i^*}^1\}$.
\end{enumerate}
We complete our proof by showing the following important claim.
\begin{claim}\label{ki}
$\mathcal{N}_{e}(P_5,G)\leq\mathcal{N}_{e}(P_5,G')$.
\end{claim}
\begin{proof}

It is enough to show that for each crossing $P_5$'s with respect to $e$ containing the edges $\{\hat{x}_{i^*},w\}$ in $G$, where $w\in I$, we correspondingly find a unique crossing $P_5$'s with respect to $e$ in $G'$ but not in $G$. We distinguish two cases considering the position of $i^*$.

\subsubsection*{Case 1: $i^*\neq p_1-1$.} 

\begin{itemize}
\item Type-I crossing $P_5$'s in $G$ containing the edges $\{\hat{x}_{i^*},w\}$: 
Type-I crossing $P_5$'s in $G$ containing $\{\hat{x}_{i^*},w\}$ are either $(\tilde{y}_{\ell},y,x,\hat{x}_{i^*},w)$ or  $(\tilde{y}_{\ell},y,x,\hat{x}_{i^*+1},\hat{x}_{i^*})$, $\ell\in[t_1]$.

For $(\tilde{y}_{\ell},y,x,\hat{x}_{i^*},w)$, we take the crossing $P_5$'s in $G'$ of the form $(\tilde{y}_{\ell},y,x,\hat{x}_{i^*}^1,w)$.

We consider two cases concerning the crossing $P_5$'s of the form $(\tilde{y}_{\ell},y,x,\hat{x}_{i^*+1},\hat{x}_{i^*})$ in $G$. If $w^*=\hat{x}_{i^*+1}$, we take the replacement crossing $P_5$ $(\tilde{y}_{\ell},y,x,\hat{x}_{i^*}^1,\hat{x}_{i^*})$. If $w^*\neq \hat{x}_{i^*+1}$, we take the replacing crossing $P_5$ $(\tilde{y}_{\ell},y,x,\hat{x}_{i^*+1},\hat{x}_{i^*}^1)$.

\item Type-II(A) crossing $P_5$'s in $G$ containing the edges $\{\hat{x}_{i^*},w\}$: It can be seen that there is no Type-II(A) crossing $P_5$'s in $G$ containing the edges $\{\hat{x}_{i^*},w\}$.

\item Type-II(B) crossing $P_5$'s in $G$ containing the edges $\{\hat{x}_{i^*},w\}$:  
In this case the only crossing $P_5$'s are of the form $(y,\tilde{x}_{s_1},x,\hat{x}_{i^*},w)$ and $(y,\tilde{x}_{s_1},x,\hat{x}_{i^*+1},\hat{x}_{i^*})$. 

For the former crossing $P_5$'s we can take the replacing crossing $P_5$'s of the form $(y,\tilde{x}_{s_1},x,\hat{x}_{i^*}^1,w)$. For the later case we consider two cases. If $w^*=\hat{x}_{i^*+1}$, then we take the replacing crossing $P_5$ $(y,\tilde{x}_{s_1},x,\hat{x}_{i^*}^1,\hat{x}_{i^*})$. If $w^*\neq \hat{x}_{i^*+1}$, we take the crossing $P_5$ of the form $(y,\tilde{x}_{s_1},x,\hat{x}_{i^*+1},\hat{x}_{i^*}^1)$. 

\item Type-III(A) crossing $P_5$'s in $G$ containing the edges $\{\hat{x}_{i^*},w\}$: 
The crossing $P_5$'s are of the form $(z,\tilde{x}_{k},x,\hat{x}_{i^*},w)$ and $(z,\tilde{x}_k,x,\hat{x}_{i^*+1},\hat{x}_{i^*})$, $k\in[s_1]$ and $z$ ($\neq y$) is a vertex adjacent to $\hat{x}_{k}$. 

For the former case we take the crossing $P_5$'s in $G'$ of the form $(z,\tilde{x}_k,x,\hat{x}_{i^*}^1,w)$.  For the latter case we consider two cases. If $w^*=\hat{x}_{i^*+1}$, we take the replacing crossing $P_5$, $(z,\tilde{x}_{k},x,\hat{x}_{i^*}^1,\hat{x}_{i^*})$. If $w^*\neq \hat{x}_{i^*+1}$, we take replacement crossing $P_5$ of the form $(z,\tilde{x}_{k},x,\hat{x}_{i^*+1},\hat{x}_{i^*}^1)$. 

\item Type-III(B) crossing $P_5$'s in $G$ containing the edges $\{\hat{x}_{i^*},w\}$: 
In this case we distinguish two cases.

\subsubsection*{Case 1.1: $i^*\neq p_1-2$.} 
The crossing $P_5$'s in $G$ containing $\{\hat{x}_{i^*},w\}$  and their respective replacement crossing $P_5$'s in $G'$ are considered below. 

\begin{enumerate}
\item For crossing $P_5$'s of the form $(\tilde{x}_k,x, \hat{x}_{i^*-1},\hat{x}_{i^*},w)$, where $k\in[s_1]$ and $w\neq w^*$, we take the replacement $(\tilde{x}_k,x, \hat{x}_{i^*},\hat{x}_{i^*}^1,w)$. Notice that when $w=w^*$, $\{\hat{x}_{i^*}^1,w^*\}$ is already an edge in $G$. In this case we replace the crossing $P_5$ $(\tilde{x}_k,x, \hat{x}_{i^*-1},\hat{x}_{i^*},w^*)$ with $(\tilde{x}_k,x, \hat{x}_{i^*}^1,\hat{x}_{i^*}^2,\hat{x}_{i^*}^3)$. Recall that there are at least two interior vertices between $\hat{x}_{i^*}$ and $\hat{x}_{i^*+1}$ in clockwise direction. If there are exactly two interior vertices, $\hat{x}_{i^*}^3=\hat{x}_{i^*+1}$.

\item For crossing $P_5$'s of the form $(\tilde{x}_k,x, \hat{x}_{i^*},w,u)$, where $k\in[s_1]$ and $u$ is a vertex adjacent to $w$ and $u\neq \hat{x}_{i^*}^1$, we replace it with by $(\tilde{x}_k,x, \hat{x}_{i^*}^1,w,u)$. For the crossing $P_5$'s of the form $(\tilde{x}_k,x, \hat{x}_{i^*},w^*,\hat{x}_{i^*}^1)$, we replace it by $(z_1,z_2,\tilde{x}_{k},x,\hat{x}_{i^*}^1)$, where $(z_1,z_2,\tilde{x}_{k})$ is a $2$-path with respect to vertex alignment in the outer cycle of $G$ in clockwise direction. 

\item For the crossing $P_5$'s of the form $(\tilde{x}_{k},x,\hat{x}_{i^*+1},\hat{x}_{i^*},u)$, where $k\in[s_1]$ and $u$ is a vertex adjacent to $\hat{x}_{i^*}$. We may consider two cases.

If $w^*= \hat{x}_{i^*+1}$, the only crossing $P_5$'s of the stated form are $(\tilde{x}_{k},x,\hat{x}_{i^*+1},\hat{x}_{i^*},\hat{x}_{i^*}^1)$, $(\tilde{x}_{k},x,\hat{x}_{i^*+1},\hat{x}_{i^*},\hat{x}_{i^*-1})$ and $(\tilde{x}_{k},x,\hat{x}_{i^*+1},\hat{x}_{i^*},\hat{x}_{i^*-1}^1)$. The last crossing $P_5$ may or may not exist, it depends with the existence of the vertex $\hat{x}_{i^*-1}^1$. The respective replacement crossing $P_5$'s we take are $(z_2,\tilde{x}_{k},x,\hat{x}_{i^*}^1,\hat{x}_{i^*}^2)$, $(\tilde{x}_{k},x,\hat{x}_{i^*}^1,\hat{x}_{i^*},\\\hat{x}_{i^*-1})$ and $(\tilde{x}_{k},x,\hat{x}_{i^*}^1,\hat{x}_{i^*},\hat{x}_{i^*-1}^1)$, where $z_2$ is the vertex as defined in(2).

On the other hand, for the case that $w^*\neq rx_{i^*+1}$,  the crossing $P_5$'s of the stated form are $(\tilde{x}_{k},x,\hat{x}_{i^*+1},\hat{x}_{i^*},w)$, $(\tilde{x}_{k},x,\hat{x}_{i^*+1},\hat{x}_{i^*},\hat{x}_{i^*}^1)$, $(\tilde{x}_{k},x,\hat{x}_{i^*+1},\hat{x}_{i^*},\hat{x}_{i^*-1})$ and $(\tilde{x}_{k},x,\hat{x}_{i^*+1},\hat{x}_{i^*},\hat{x}_{i^*-1}^1)$. The last crossing $P_5$ may or may not exist. In this case we take the the replacement crossing $P_5$'s, $(\tilde{x}_{k},x,\hat{x}_{i^*+1},\hat{x}_{i^*}^1,w)$, $(\tilde{x}_{k},x,\hat{x}_{i^*+1},\\\hat{x}_{i^*}^1,\hat{x}_{i^*})$, $(\tilde{x}_{k},x,\hat{x}_{i^*}^1,\hat{x}_{i^*},\hat{x}_{i^*-1})$ and $(\tilde{x}_{k},x,\hat{x}_{i^*}^1,\hat{x}_{i^*},\hat{x}_{i^*-1}^1)$ respectively.  



\item For the case that $w^*\neq \hat{x}_{i^*+1}$, we have a crossing $P_5$ of the form $(\tilde{x}_{k},x, \hat{x}_{i^*+1},f,\\\hat{x}_{i^*})$, where $f$ is the second from the last $\hat{x}_{i^*+1}$ in $I$ such that $\{\hat{x}_{i^*},f\}$ is an edge in $G$. In this case we take replacement crossing $P_5$ $(z_3,z_4,\tilde{x}_{k},x,\hat{x}_{i^*}^1)$. where $(z_3,z_4,\tilde{x}_{k})$ is a $2$-path with respect to vertex alignment in the outer cycle of the $G$ in counterclockwise direction.

\item For the crossing $P_5$ of the form $(\tilde{x}_k, x, \hat{x}_{i^*+2},\hat{x}_{i^*+1},\hat{x}_{i^*})$. We distinguish two cases. If $w*=\hat{x}_{i^*+1}$, then we take the replacement crossing $P_5$ $(z_4,\tilde{x}_k,x,\hat{x}_{i^*}^1,\hat{x}_{i^*}^2)$. If $w^*\neq \hat{x}_{i^*+1}$, we take the replacement crossing $P_5$ $(\tilde{x}_k,x,\hat{x}_{i^*+2},\hat{x}_{i^*+1},\hat{x}_{i^*}^1)$.
\end{enumerate}

\subsubsection*{Case 1.2: $i^*= p_1-2$.}
Apart from the crossing $P_5$'s considered above, the crossing $P_5$'s which is not considered yet are Type-III(B) crossing $P_5$ of the form $(\tilde{y}_{\ell},y,\hat{x}_{p_1},\hat{x}_{p_1-1},\hat{x}_{p_1-2})$, where $\ell\in[t_1]$. Here we distinguish two cases. If $w^*=\hat{x}_{p_1-1}$, we take the replacing crossing $P_5$, $(\tilde{y}_{\ell},y,x,\hat{x}_{p_1-2}^1,\hat{x}_{p_1-2})$. If $w^*\neq \hat{x}_{p_1-1}$, we take the replacing crossing $P_5$ $(\tilde{y}_{\ell},y,\hat{x}_{p_1},\hat{x}_{p_1-1},\hat{x}_{p_1-2}^1)$.
\end{itemize}
\subsubsection*{Case 2: $i^*= p_1-1$.}  For completeness we consider crossing $P_5$'s of each type containing the edges $\{\hat{x}_{p_1-1},w\}$, where $w\in I$.
\begin{itemize}
\item Type-I crossing $P_5$'s in $G$ containing the edges $\{\hat{x}_{p_1-1},w\}$: Apart from crossing Type-I crossing $P_5$'s considered in Case 1, the Type-I crossing $P_5$ which is not considered in this case is $(\tilde{x}_{k},x,y,\hat{x}_{p_1},\hat{x}_{p_1-1})$. We consider two cases. If $w^*=\hat{x}_{p_1}$, then we take the replacement crossing $P_5$'s $(z_4,\hat{x}_{k},x,\hat{x}_{p_1-1}^1,\hat{x}_{p_1-1}^2)$, where the $z_4$ is the vertex whose definition is given in Case 1.1 (4). If $w^*\neq \hat{x}_{p_1}$, we take the replacement crossing $P_5$ of the form $(\tilde{x}_{k},x,y,\hat{x}_{p_1},\hat{x}_{p_1-1}^1)$.

\item Type-II(A) crossing $P_5$'s in $G$ containing the edges $\{\hat{x}_{p_1-1},w\}$: It can be seen that there is no Type-II(A) crossing $P_5$'s in $G$ containing the edges $\{\hat{x}_{p_1-1},w\}$.

\item Type-II(B) crossing $P_5$'s in $G$ containing the edges $\{\hat{x}_{p_1-1},w\}$: The crossing $P_5$'s of this kind which are not yet considered are $(\tilde{x}_{k},x,\hat{x}_{p_1-1},\hat{x}_{p_1},y)$ and $(\tilde{y}_{t},y,\hat{x}_{p_1},\hat{x}_{p_1-1},x)$,  where $k\in[s_1]$ and $\ell\in[t_1]$. In this case we take the replacement crossing $P_5$'s of the form $(\tilde{x}_{k},x,\hat{x}_{p_1-1}^1,\hat{x}_{p_1},y)$ and $(\tilde{y}_{t},y,\hat{x}_{p_1},\hat{x}_{p_1}^1,x)$ respectively.

\item Type-III(A) crossing $P_5$'s in $G$ containing the edges $\{\hat{x}_{p_1-1},w\}$: Apart from the crossing $P_5$'s of Type III(A) stated in Case 1, the following are the new Type-III(A) crossing $P_5$'s which are not yet considered. These are $(v,\tilde{y}_{\ell},y,\hat{x}_{p_1},\hat{x}_{p_1-1})$, where $\ell\in[t_1]$. We consider two cases. If $w^*=\hat{x}_{p_1}$, we take the replacement crossing $P_5$ of the form $(v,\tilde{y}_{\ell},y,x,\hat{x}_{p_1-1}^1)$. If $w^*\neq \hat{x}_{p_1}$, we take the replacement crossing $P_5$'s of the form $(v,\tilde{y}_{\ell},y,\hat{x}_{p_1},\hat{x}_{p_1-1}^1)$. 

\item Type-III(B) crossing $P_5$'s in $G$ containing the edges $\{rx_{p_1-1},w\}$:
In this case we do not have crossing $P_5$'s which are stated in Case 1.1(5). However we have new Type-III(B) crossing $P_5$'s. The following are these crossing $P_5$'s and their respective replacements.
\begin{enumerate}
\item Consider the crossing $P_5$'s of the form $(\tilde{y}_{\ell},y,\hat{x}_{p_1},\hat{x}_{p_1-1},u)$. We distinguish two cases. If $w^*=\hat{x}_{p_1}$, then there are only three crossing $P_5$'s of such kind. These are $(\tilde{y}_{\ell},y,\hat{x}_{p_1},\hat{x}_{p_1-1},\hat{x}_{p_1-1}^1)$, $(\tilde{y}_{\ell},y,\hat{x}_{p_1},\hat{x}_{p_1-1},\hat{x}_{p_1-2})$ and $(\tilde{y}_{\ell},y,\hat{x}_{p_1},\hat{x}_{p_1-1},\hat{x}_{p_1-2}^1)$. Notice that the existence of the last crossing $P_5$ depends on the existence of the vertex $\hat{x}_{p_1-2}^1$. We take the following respective replacement crossing $P_5$'s, $(u_1,\tilde{y}_{\ell},y,x,\\\hat{x}_{p_1-1}^1)$, $(u_2,\tilde{y}_{\ell},y,x,\hat{x}_{p_1-1}^1)$ and $(\tilde{y}_{\ell},y,x,\hat{x}_{p_1-1}^1,\hat{x}_{p_1-1}^2)$ respectively. It can be seen that non of these crossing $P_5$'s is used in any the previous crossing $P_5$'s. 

If $w^*=\hat{x}_{p_1}$, then the crossing $P_5$'s of such kind are $(\tilde{y}_{\ell},y,\hat{x}_{p_1},\hat{x}_{p_1-1},w)$, $(\tilde{y}_{\ell},y,\hat{x}_{p_1},\\\hat{x}_{p_1-1},\hat{x}_{p_1-1}^1)$, $(\tilde{y}_{\ell},y,\hat{x}_{p_1},\hat{x}_{p_1-1},\hat{x}_{p_1-2})$ and $(\tilde{y}_{\ell},y,\hat{x}_{p_1},\hat{x}_{p_1-1},\hat{x}_{p_1-2}^1)$. In this case we take the replacement crossing $P_5$'s of the form $(\tilde{y}_{\ell},y,\hat{x}_{p_1},\hat{x}_{p_1-1}^1,w)$, $(\tilde{y}_{\ell},y,\hat{x}_{p_1},\hat{x}_{p_1-1},\\\hat{x}_{p_1-1})$ $(u_1,\tilde{y}_{\ell},y,x,\hat{x}_{p_1-1}^1)$ and $(u_2,\tilde{y}_{\ell},y,x,\hat{x}_{p_1-1}^1)$. 

\item For the case that $w^*\neq \hat{x}_{p_1}$, we have a crossing $P_5$'s of the form $(\tilde{y}_{\ell},y,\hat{x}_{p_1},f,\\\hat{x}_{p_1-1})$, where $f$ is second vertex from the last $\hat{x}_{p_1}$ in $I$.
In this case we take the replacement crossing $P_5$ of the form $(\tilde{y}_{\ell},y,\tilde{x}_{s_1},x,\hat{x}_{p_1-1}^1)$.

\item Consider the crossing $P_5$'s of the form $(\tilde{y}_{\ell},y,\hat{y}_{q_1-1},\hat{x}_{p_1},\hat{x}_{p_1-1})$. In this case we take the replacement $(\tilde{y}_{\ell},y,\hat{x}_{p_1},x,\hat{x}_{p_1-1}^1)$.
\end{enumerate}
\end{itemize}

Therefore the number of crossing $P_5$'s with respect to $e$ in $G$ is at most the number of crossing $P_5$'s with respect to $e$ in $G'$. 
This completes the proof of Claim~\ref{ki}.\qedhere
\end{proof}

From Claim~\ref{ki}, we can proceed with the reduction operation procedure for each neighbours of $x$ and $y$ in $I_1$ and $I_2$ till we reach on the situation that, for each $i\in[p_1-1], j\in[q_1-1], k\in[s_1-1]$ and $\ell\in[t_1-1]$, there is no interior or one interior vertex between the pairs $(\tilde{x}_{k}, \tilde{x}_{k+1})$ in counterclockwise direction, $(\tilde{y}_{\ell}, \tilde{y}_{\ell+1})$ in clockwise direction, $(\hat{x}_{i},\hat{x}_{i+1})$ in clockwise direction and $(\hat{y}_{j},\hat{y}_{j+1})$ counterclockwise direction. 

Let the resulting graph at the end of a sequence of reduction operations be $H$. Thus $\mathcal{N}_{e}(P_5,G)\leq \mathcal{N}_{e}(P_5,H).$  But by Claim~\ref{cv1}, 
 $\mathcal{N}_{e}(P_5,H)=\frac{17}{2}n_1n_2+O(n)$. Therefore,  $$\mathcal{N}_{e}(P_5,G)=\frac{17}{2}n_1n_2+O(n).$$ 

This completes the proof of Lemma~\ref{mainlema}.\qedhere
\end{proof}

Next we finish the proof of Theorem~\ref{tms2}. Let the graph induced by $I_1\cup\{x,y\}$ and $I_2\cup \{x,y\}$ be $G_1$ and $G_2$ respectively. Since $e$ is a nice chord, by induction $\mathcal{N}(P_5,G_1)=\frac{17}{4}n_1^2+O(n)$ and  $\mathcal{N}(P_5,G_2)=\frac{17}{4}n_2^2+O(n)$. From Lemma~\ref{mainlema}, $\N_{e}(P_5,G)=\frac{17}{2}n_1n_2+O(n).$

Therefore,
\begin{align*}
\mathcal{N}(P_5,G)&=\mathcal{N}(P_5,G_1)+\mathcal{N}(P_5,G_2)+\mathcal{N}_{e}(P_5,G)\\&=\frac{17}{4}n_1^2+O(n)+\frac{17}{4}n_2^2+O(n)+\frac{17}{2}n_1n_2+O(n)\\&=\frac{17}{4}(n_1+n_2)^2+O(n)\\&=\frac{17}{4}n^2+O(n).
\end{align*}
This completes the proof of Theorem~\ref{tms2}.\qedhere
\end{proof}
\section{Concluding remarks and conjectures}
Considering the complexity of the proof we have for a best asymptotic value of $f_{\mathcal{OP}}(n,P_5)$, it might be not easy to determine a best asymptotic value of the generalized outerplanar Tur\'an number of short paths. We pose the following conjecture related to the generalized outerplanar Tur\'an number of the $P_6$.
\begin{conjecture}
$f_{\mathcal{OP}}(n,P_6)=11n^2+\Theta(n)$.
\end{conjecture}
The following construction of an $n$-vertex maximal outerplanar graph $G_n$ verifies the lower bound is attainable. $G_n$ contains $\frac{n}{2}$ degree-2 vertices and all the remaining $\frac{n}{2}-1$ vertices are adjacent to a vertex, say $v$, in $G_n$ (see Figure~\ref{fig2sr}). It can be checked that $\N(P_6, G_n)=11n^2 +\Omega(n).$ 

\begin{figure}[ht]
\centering
\begin{tikzpicture}[scale=0.2]
\draw[thick](0,20)--(10,-17.3)(0,20)--(17.3,-10)(0,20)--(20,0)(0,-20)--(0,20)(0,20)--(17.3,10)(0,20)--(10,17.3);
\draw[thick](0,20)--(-10,-17.3)(0,20)--(-17.3,-10)(0,20)--(-20,0)(0,-20)--(0,20)(0,20)--(-17.3,10)(0,20)--(-10,17.3);
\draw[black,thick](10,17.3)--(17.3,10)--(20,0)--(17.3,-10)--(10,-17.3)--(0,-20);
\draw[black,thick](-10,17.3)--(-17.3,10)--(-20,0)--(-17.3,-10)--(-10,-17.3)--(0,-20);
\draw[fill=black] (19.3,5.2) circle (15pt);
\draw[fill=black] (-19.3,5.2) circle (15pt);
\draw[fill=black] (-19.3,-5.2) circle (15pt);
\draw[fill=black] (19.3,-5.2) circle (15pt);
\draw[fill=black] (14,14) circle (15pt);
\draw[fill=black] (-14,14) circle (15pt);
\draw[fill=black] (-14,-14) circle (15pt);
\draw[fill=black] (14,-14) circle (15pt);
\draw[fill=black] (5.2,19.3) circle (15pt);
\draw[fill=black] (-5.2,19.3) circle (15pt);
\draw[fill=black] (-5.2,-19.3) circle (15pt);
\draw[fill=black] (5.2,-19.3) circle (15pt);
\draw[thick] (0,0) circle (20cm);
\draw[fill=black] (20,0) circle (15pt);
\draw[fill=black] (-20,0) circle (15pt);
\draw[fill=black] (0,20) circle (15pt);
\draw[fill=black] (0,-20) circle (15pt);
\draw[fill=black] (17.3,10) circle (15pt);
\draw[fill=black] (10,17.3) circle (15pt);
\draw[fill=black] (-17.3,10) circle (15pt);
\draw[fill=black] (-10,17.3) circle (15pt);
\draw[fill=black] (-17.3,-10) circle (15pt);
\draw[fill=black] (-10,-17.3) circle (15pt);
\draw[fill=black] (17.3,-10) circle (15pt);
\draw[fill=black] (10,-17.3) circle (15pt);
\node at (0,23) {$v$};
\end{tikzpicture}
\caption{A maximal outerplanar graph $G_n$ containing roughly $11n^2$ $P_6$'s.}
\label{fig2sr}
\end{figure}
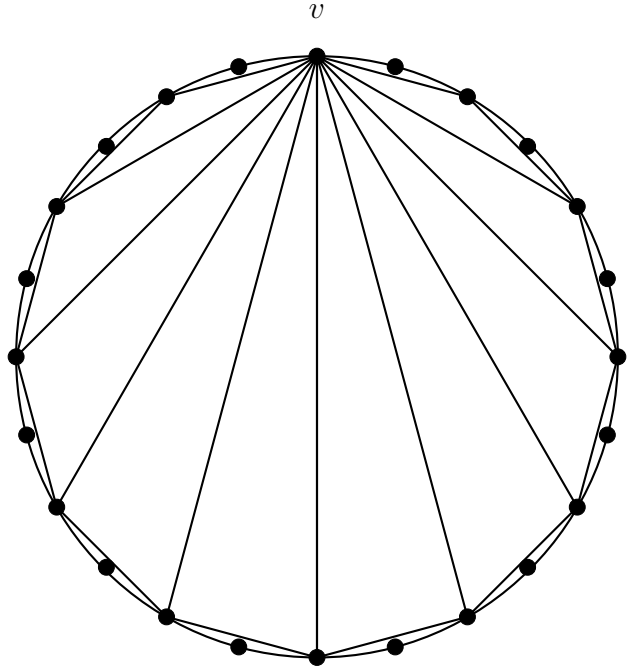
\section*{Acknowledgments}
The research of all the authors was partially supported by the National Research, Development and Innovation Office NKFIH, grants  K132696, and K126853.

\section*{Dedications}
The \textbf{second author} would like to dedicate this work to the memory of all innocent \textbf{Ethiopians} at every corner of the country who lost their life as a result of the \textbf{ethnic politics and hatreds} in the past \textbf{four dark years}.
\begin{figure}[ht]
\centering
\includegraphics[scale=0.5]{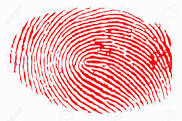}
\end{figure}

\section*{Conflicts of interest}
The authors declare that they have no conflicts of interest.


\end{document}